\documentclass[
	english,
]{siamart250211}

\usepackage[T1]{fontenc}
\usepackage[utf8]{inputenc}

\usepackage{amsfonts}

\usepackage{preamble}
\usepackage{booktabs,multirow}

\usepackage{graphbox}

\usepackage{enumitem}
\setenumerate{label=(\roman*)}

\numberwithin{equation}{section} 

\usepackage{lmodern}
\usepackage[
	final,
	nopatch=eqnum 
]{microtype}

\allowdisplaybreaks

\usepackage{mathtools}

\providecommand\given{\nonscript\;\delimsize|\nonscript\;\mathopen{}}
\DeclarePairedDelimiterX\set[1]\{\}{#1}
\DeclarePairedDelimiterX\seq[1](){#1}

\DeclarePairedDelimiterX\dual[2]{\langle}{\rangle}{#1,#2}
\DeclarePairedDelimiterX\innerprod[2](){#1,#2}

\DeclarePairedDelimiter\abs{\lvert}{\rvert}
\DeclarePairedDelimiter\norm{\lVert}{\rVert}

\DeclarePairedDelimiter\parens()

\newcommand\N{\mathbb{N}}
\newcommand\R{\mathbb{R}}

\renewcommand\d{\mathop{}\!\mathrm{d}}
\newcommand\dx{\d x}
\newcommand\dt{\d t}

\newcommand\HH{\mathcal{H}}

\newcommand\MM{\mathcal{M}}

\DeclareMathOperator{\cl}{cl}
\DeclareMathOperator\interior{int}
\DeclareMathOperator\divergence{div}
\DeclareMathOperator\conv{conv}
\DeclareMathOperator\diag{diag}

\DeclareMathOperator\supp{supp}
\newcommand{\smallO}{\mathpzc{o}}
\newcommand\Uad{U_{\mathrm{ad}}}

\DeclareMathAlphabet{\mathpzc}{OT1}{pzc}{m}{it}
\newcommand\oo{\mathpzc{o}}

\newsiamremark{remark}{Remark}
\AddToHook{env/remark/begin}{\crefalias{theorem}{remark}}

\begin{document}
\title{
	Numerical solution of
	optimal control problems
	using quadratic transport regularization
}

\author{%
	Nicolas Borchard%
	\thanks{%
		Brandenburgische Technische Universität Cottbus-Senftenberg,
		Institute of Mathematics,
		03046 Cottbus, Germany,
		\email{nicolas.borchard@b-tu.de}%
	}
	\orcid{0009-0007-7358-7737}%
	\and
	Gerd Wachsmuth%
	\thanks{%
		Brandenburgische Technische Universität Cottbus-Senftenberg,
		Institute of Mathematics,
		03046 Cottbus, Germany,
		\email{gerd.wachsmuth@b-tu.de}%
	}
	\orcid{0000-0002-3098-1503}%
}

\maketitle

\begin{abstract}
	We address optimal control problems on the space of measures for an objective containing a
	smooth functional and an optimal transport regularization.
	That is,
	the quadratic Monge-Kantorovich distance between a given prior measure
	and the control
	is penalized in the objective.
	We consider optimality
	conditions and reparametrize the problem using the celebrated structure theorem by Brenier.
	The optimality conditions can be formulated as a piecewise differentiable equation.
	This is utilized to formulate solution algorithms
	and to analyze their local convergence properties.
	We present a numerical example to illustrate the theoretical findings.
\end{abstract}

\begin{keywords}
	optimal transport regularization,
	semismooth Newton method,
	measure control
\end{keywords}

\begin{MSCcodes}
	49M15,
	49M29,
	49K20
\end{MSCcodes}

\section{Introduction}

We are interested in optimal control problems
governed by a partial differential equation (PDE)
in which the control is given by a measure.
In the literature,
there are many works which
address this topic, we refer exemplarily to
\cite{ClasonKunisch2011,CasasClasonKunisch2012,PieperVexler2013}.
In these contributions,
the Radon norm is used to regularize the control.
In contrast, we consider the regularization by
adding a quadratic Monge-Kantorovich distance to a given prior measure $u_d$.
This distance measures the transport costs if the given measure $u_d$
is transported (in an optimal way) to the control $u$.
We refer to \cite{Villani2003,Santambrogio,FigalliGlaudo2023}
for an introduction to the field of optimal transport.

To motivate this class of problems,
we consider a situation in which the given measure $u_d$
describes the spatial distribution of a certain resource.
It should be transported to some new, unknown position modeled
by a measure $u$.
Consequently, $u$ influences some physical process described by an operator $S$.
The outcome $S(u)$ should be close to a given, desired state $y_d$
and this has to be balanced in an optimal way with the transport costs.

This leads to the problem
\begin{equation*}
	\label{eq:OCP}
	\tag{OCP}
	\text{Minimize} \quad \frac12 \norm{S(u) - y_d}_{L^2(O)}^2 + \frac\alpha2 W_2^2(u_d, u)
	\quad
	\text{with respect to } u \in \mathcal{M}(F)
	.
\end{equation*}
Here,
$S \colon \mathcal{M}(F) \to L^2(O)$
is the control-to-state map
and
$W_2^2(u_d, u)$ is the squared quadratic Monge-Kantorovich distance
between the given measure $u_d \in \mathcal{M}(D)$
and the control $u \in \mathcal{M}(F)$,
see \eqref{eq:monge_kantorovich} below.
Concerning the given prior,
we assume that $u_d$ is absolutely continuous
w.r.t.\ the Lebesgue measure.
Moreover, $y_d \in L^2(O)$ is a given desired state
and $\alpha > 0$
balances the tracking term and the transport costs.
The sets $O, F, D$ are subsets of $\R^2$.

As a concrete example,
let $u_d$ be the distribution of some waste.
After it has been transported to its new position $u$,
the waste enters the convection-diffusion equation
\begin{equation}
	\label{eq:PDE}
	\mathopen{} 
	-\Delta y + \beta \cdot \nabla y = u \; \text{in }\Omega,
	\qquad
	y = 0 \; \text{on }\partial\Omega.
\end{equation}
Here, $y$ models the concentration of some chemicals
(exhaled from the distribution $u$ of the waste)
in the air.
The spread of these chemicals is subject to diffusion
and convection due to the wind speed represented by the vector field $\beta$. 
The goal is to minimize
the concentration $y$ in a certain critical area $O$, i.e., $y_d \equiv 0$.
Finally, the domain $\Omega \subset \R^2$ is chosen large enough
such that the (artificial) boundary conditions
do not have a large impact on the resulting $y$ in the observation area $O$.

The aim of this paper is the numerical solution of problems of the type \eqref{eq:OCP}.
In general, this is an infinite-dimensional optimization problem.
Hence, it has to be discretized at some point in order to be solved numerically.
We propose to use a first-discretize-then-optimize approach, i.e.,
we discretize the control variable $u$.
A possible discretization of $u$ is given by a linear combination of Dirac measures
at fixed points $a_1, \ldots, a_n \in \R^2$.
After this discretization,
we still have a problem of type \eqref{eq:OCP}
with the choice $F := \set{a_1, \ldots, a_n}$.
Note that this implies that the space of measures $\mathcal{M}(F)$
can be identified with $\R^n$,
i.e.,
the vector $v \in \R^n$ is identified with the linear combination
$u = \sum_{i = 1}^n v_i \delta_{a_i} \in \mathcal{M}(F)$ of Dirac measures.
Consequently,
\eqref{eq:OCP}
becomes a finite-dimensional optimization problem.

For the following discussion,
it is beneficial to write \eqref{eq:OCP} in the slightly more general form
\begin{equation*}
	\tag{P}
	\label{P}
	\text{Minimize} \quad
	g(u) + \frac \alpha 2 W_2^2(u_d, u)
	\quad
	\text{with respect to } u \in \mathcal{M}(F),
\end{equation*}
with objective function
$g \colon \mathcal{M}(F) \to \R$
or $g \colon \R^n \to \R$, owing to the identification $\mathcal{M}(F) \cong \R^n$.

Since $u$ is a discrete measure and since $u_d$
is assumed to be absolutely continuous,
the computation of $W_2^2(u_d, u)$
(for a fixed $u$)
is a so-called semi-discrete optimal transport problem.
Numerical methods for such problems are given in, e.g.,
\cite{Merigot2011,KitagawaMerigotThibert2019,TaskesenShafieezadehAbadehKuhn2022}.
In particular,
\cite{KitagawaMerigotThibert2019}
proves global convergence of a damped Newton method.
This being said,
even the evaluation of $W_2^2(u_d, u)$
(for fixed $u$)
is computationally expensive,
see the comment after
\eqref{eq:eval_transport_dist},
and this renders \eqref{P} a challenging problem,
see also the discussion at the beginning of \cref{sec:algorithms}.

We are not aware of any contributions
concerning solution methods
which are directly applicable to \eqref{P}.
Closest to our work are
\cite{Bansil2021,BansilKitagawa2021}.
Therein, the authors study numerical methods for \eqref{P}
with a structured $g$.
In \cite{BansilKitagawa2021},
$g$ is assumed to be the indicator function of a box in $\R^n$,
i.e., this is a constraint on the coefficient vector $v$ of the measure $u$.
In \cite{Bansil2021},
the function $g$ has to be separable,
i.e., $g(v) = \sum_{i = 1}^n g_i(v_i)$
and the component functions $g_i$ have to possess a very special structure,
see \cite[Theorem~2.7]{Bansil2021}.
In both papers,
the authors prove that a Newton-like method converges globally
and locally with a superlinear rate.
In \cite{BansilKitagawa2021}, some numerical results are presented,
but no numerical experiments were performed in \cite{Bansil2021}.
The case $g \equiv 0$ is studied in
\cite{DeGournayKahnLebrat2018,LebratDeGournayKahnWeiss2019},
but the underlying set of measures $u$ is more general,
i.e.,
the positions $a_i$ itself are optimization variables
and there might be constraints on the positions $a_i$
and on the masses $v_i$ in the measure $u$,
see \cite[(2.1), (2.2)]{LebratDeGournayKahnWeiss2019}.
Finally,
we mention that in all of these papers, more general transport costs were considered.

As already said,
our goal is the design
of efficient numerical methods for the solution of
\eqref{eq:OCP} and \eqref{P}.
To this end,
theoretical considerations are necessary.
In particular,
our main contributions are
\begin{enumerate}
	\item
		reformulation of the optimality conditions as $r(\bar \xi) = 0$
		with $r \colon \R^n \to \R^n$, see \cref{lem:optimality of xi on J},
	\item
		verification of semismoothness of $r$,
		see \cref{thm:J_r_tilde_u_W_PC1_semismooth},
	\item
		convergence results for a fixed-point algorithm,
		see \cref{thm:conv_ODE,thm:gradient_method_constant_stepsize_R_linear},
		and for a semismooth Newton method,
		see \cref{thm:convergence of Newton method on r}.
\end{enumerate}

Our findings can also be applied to
semi-discrete optimal transport problems with storage fees, see \cite{Bansil2021,BansilKitagawa2021},
and to
semi-discrete optimal transport problems with queue penalization, see \cite{CrippaJimenezPratelli2009},
if the data of these problems is regular enough
such that they can be formulated as \eqref{P}
with a twice continuously differentiable function $g$.

The outline of the paper is as follows.
\Cref{sec:preliminaries} addresses preliminaries.
In particular,
we fix some notations and the standing assumptions (\cref{subsec:notation_assumptions}),
provide optimality conditions for \eqref{P} (\cref{subsec:existence_and_optimality}),
and recall concepts of generalized differentiability (\cref{subsec:concepts_of_differentiability}).
In \cref{sec:properties_Pxi}
we investigate a reformulation \eqref{P(xi)} of \eqref{P}.
In particular, we provide the semismoothness of the involved functions, see
\cref{thm:J_r_tilde_u_W_PC1_semismooth}.
In combination with the optimality conditions from \cref{subsec:existence_and_optimality},
this gives rise to the algorithms studied in \cref{sec:algorithms}.
We consider a fixed-point iteration and a semismooth Newton method
and provide their local convergence,
see \cref{subsec:fixed_point}
and \cref{subsec:Newton on r}, respectively.
Finally,
the numerical experiments of \cref{sec:numerics}
illustrate the theoretical findings.

\section{Preliminaries}
\label{sec:preliminaries}

\subsection{Notation and assumptions}
\label{subsec:notation_assumptions}
Throughout the paper,
we denote by $\mathcal{M}(B)$
the set of (signed) Borel measures on a compact set $B \subset \R^2$
and $\delta_b$ is the Dirac measure of a point $b \in \R^2$.
By the Riesz representation theorem,
we have that $\mathcal{M}(B)$ is the dual space
of the set of continuous functions $C(B)$.

For convenience,
we summarize our assumptions
on problem \eqref{P}.
We assume
\begin{enumerate}[label=(A\arabic*)]
	\item
		$F = \set{a_1, \ldots, a_n} \subset \R^2$
		is a finite set and all $a_i$ are pairwise distinct,
	\item
		$D \subset \R^2$ is a convex, compact polygon with nonempty interior,
		i.e., a full-dimensional (convex) polytope,
	\item
		\label{item:g_C2}
		$g \colon \mathcal{M}(F) \to \R$ is twice continuously differentiable,
	\item
		$\alpha > 0$ and $W_2^2$ is the squared quadratic Monge-Kantorovich distance, see \eqref{eq:monge_kantorovich} below,
	\item
		$u_d \in \mathcal{M}(D)$ is a given nonnegative measure
		which is absolutely continuous w.r.t.\ the Lebesgue measure
		and its density $\varrho$ is assumed to be continuous on $D$.
\end{enumerate}
The squared quadratic Monge-Kantorovich distance
$W_2^2(u_d, u)$
between the given measure $u_d \in \mathcal{M}(D)$
and the control $u \in \mathcal{M}(F)$
is defined via
\begin{equation}
	\label{eq:monge_kantorovich}
	W_2^2(u_d, u)
	:=
	\inf\set*{
		\int_{D \times F} \abs{ x_1 - x_2 }^2 \d\gamma(x_1, x_2)
		\given
		\gamma \in \Gamma( u_d, u )
	}
	,
\end{equation}
where
the set $\Gamma(u_d, u)$
of couplings between $u_d$ and $u$
is given by
\begin{equation*}
	\Gamma( u_d, u )
	:=
	\set{
		\gamma \in \mathcal{M}( D \times F )
		\given
		\gamma \ge 0, \;
		\pi_1 \mathbin{\#} \gamma = u_d, \;
		\pi_2 \mathbin{\#} \gamma = u
	}
	.
\end{equation*}
Here, $\pi_1 \colon D \times F \to D$
and $\pi_2 \colon D \times F \to F$ are the projections
and ``$\#$'' denotes the push-forward of measures.

Note that the definition of the transport distance
yields
that $W_2^2(u_d, u) < \infty$
implies
the nonnegativity $u \ge 0$
and the equality of total masses $u(F) = u_d(D)$.
Therefore,
\eqref{P} implicitly includes
the constraint
\begin{equation*}
	u
	\in
	\Uad
	:=
	\set*{u \in \mathcal{M}(F) \given u \ge 0, \; u(F) = u_d(D)}
	.
\end{equation*}
Since $F$ is a finite set,
the space of measures $\mathcal{M}(F)$
is finite dimensional
and we identify it with $\R^n$.
Consequently,
$\Uad = \set*{\sum_{i=1}^n v_i \delta_{a_i} \given v_i \ge 0, \sum_{i=1}^n v_i = u_d(D)}$
is identified
with the set of vectors
$\set{ v \in \R^n \given v \ge 0, \sum_{i = 1}^n v_i = u_d(D)}$.
Similarly,
we also write $g(v)$ and $W_2^2(u_d, v)$ for a vector $v$ from this set.
Note that assumption \ref{item:g_C2} is equivalent to $g \in C^2(\R^n)$
under this identification.

For dealing with the problem \eqref{eq:OCP},
we set $g(u) := \frac12 \norm{S(u) - y_d}_{L^2(O)}^2$.
If the solution operator $S \colon \MM(F) \to L^2(O)$
is twice continuously differentiable,
we get
$g \in C^2(\mathcal{M}(F))$ via the chain rule.

Using standard arguments,
see, e.g., \cite[Section~2]{ClasonKunisch2011},
one can show (under appropriate assumptions on the data)
that the PDE \eqref{eq:PDE}
leads to a suitable solution operator $S \colon \MM(F) \to L^2(O)$.
In particular, this requires some regularity of the domain $\Omega$,
e.g., the boundary is $C^{1,\delta}$, $\delta \in (0,1]$, or polygonal.

\subsection{Existence of solutions and optimality conditions}
\label{subsec:existence_and_optimality}
First, we show that problem \eqref{P} possesses a solution.
\begin{theorem}
	\label{thm:uniqueness}
	There exists a global solution of \eqref{P}. If $g$ is convex, the solution is unique.
\end{theorem}
\begin{proof}
	Obviously,
	$\Uad$ is the image of the compact set
	$\set{v \in \R^n \given v \ge 0, \sum_{i=1}^n v_i = u_d(D)}$
	under the continuous map
	$\R^n \ni v \mapsto \sum_{i=1}^n v_i \delta_{a_i} \in \mathcal{M}(F)$.
	Consequently, $\Uad$ is compact in $\mathcal{M}(F)$.
	Weierstraß' theorem yields a minimizer as
	the objective functional is continuous.

	Now let $g$ additionally be convex.
	Since $u_d$ is absolutely continuous w.r.t.\ the Lebesgue measure,
	the functional $W_2^2(u_d, \cdot)$ is strictly convex,
	see \cite[Proposition~7.19]{Santambrogio}.
	Consequently,
	the objective of \eqref{P} is strictly convex.
	Hence, the minimizer is unique.
\end{proof}
Next, we address first-order optimality conditions.
Since the set $F$ is finite,
the space $\mathcal{M}(F)$ is finite dimensional.
This allows us to identify the derivative $g'(u)$ for $u \in \mathcal{M}(F)$
with a continuous function from $C(F)$.
For a function $\psi \in C(F)$, we define its $c$-conjugate
$\psi^c \colon \R^2 \to \R$
via
\begin{equation*}
	\psi^c(x_1)
	:=
	\inf\set{ c(x_1,x_2) - \psi(x_2) \given x_2 \in F }
	=
	\inf\set{ c(x_1,a_i) - \psi(a_i) \given i = 1,\ldots,n }
	.
\end{equation*}
Note that this definition easily implies that $\psi^c$ is continuous.

In the next lemma,
we investigate the convex (pre)-conjugate
of the transport term from \eqref{P}.
Recall that $F$ is a finite set,
therefore the spaces $C(F)$ and $\MM(F)$
are finite-dimensional,
dual to each other and, thus, reflexive.

\begin{lemma}
	\label{lem:h for predual}
	The function $H \colon C(F) \to \R$ defined via
	$H(\varphi) := -\int_D \varphi^c \d u_d$
	is
	proper, convex and continuous.
	Moreover, the convex conjugate
	$H^*$ of $H$
	satisfies
	$H^* = W_2^2(u_d, \cdot)$.
\end{lemma}
\begin{proof}
	Since $c$-conjugates are continuous,
	$H(\varphi) \in \R$ follows for all $\varphi \in C(F)$.
	Further, one can check that
	\begin{equation*}
		( \lambda \varphi_1 + (1 - \lambda) \varphi_2 )^c
		\ge
		\lambda \varphi_1^c
		+
		(1 - \lambda) \varphi_2^c
		,
	\end{equation*}
	see the proof of 
	\cite[Prop.~7.17]{Santambrogio}.
	This implies that $H$ is convex
	and, consequently, continuous.
	Further, for $u \in \mathcal{M}(F)$, it holds
	\[
		H^*(u)
		=
		\sup_{\varphi \in C(F)}
		\parens*{
			\int_F \varphi \d u + \int_D \varphi^c \d u_d
		}
		=
		W_2^2(u_d, u)
		.
	\]
	In case of $u \in \Uad$,
	the last equality
	is the famous Kantorovich duality, see
	\cite[Thm.~1.39]{Santambrogio}.
	In case $u \not\in \Uad$,
	one can easily check that the supremum
	evaluates to $\infty$
	and $W_2^2(u_d, u) = \infty$ as well.
\end{proof}

\begin{lemma}
	\label{lem:optimality condition of u on G}
	Let $\bar u$ be a local minimizer of \eqref{P}.
	Then,
	\begin{equation}
		\label{eq:optimality}
		\mathopen{} 
		-2g'(\bar u)/\alpha
		\in
		\set*{\psi \in C(F) \given \int_D \psi^c \d u_d + \int_F \psi \d \bar u = W_2^2(u_d,\bar u)}
		.
	\end{equation}
	If $g$ is convex,
	\eqref{eq:optimality} is sufficient for global optimality.
\end{lemma}
\begin{proof}
	As $g$ is continuously differentiable and
	$W_2^2$ is convex,
	the optimality condition is
	$0 \in g'(\bar u) + \frac \alpha 2 \partial (W_2^2(u_d,\cdot))(\bar u)$.
	Rearranging this condition yields
	\begin{equation*}
		-2g'(\bar u) / \alpha
		\in
		\partial (W_2^2(u_d,\cdot))(\bar u)
		=
		\partial H^*(\bar u)
	\end{equation*}
	using the function $H$ from \cref{lem:h for predual}.
	From the characterization
	\begin{equation*}
		\partial H^*(\bar u)
		=
		\set{
			\varphi \in C(F)
			\given
			H^*(\bar u) + H(\varphi)
			=
			\dual{\bar u}{\varphi}
		}
	\end{equation*}
	of the subdifferential,
	we get the first claim.

	If $g$ is convex, it is clear that \eqref{eq:optimality}
	yields optimality.
\end{proof}

Brenier's theorem says that the optimal transport
from $u_d$ to $\bar u$
can be realized (under
some conditions) using a transport map which is the gradient of a
convex function $\varphi$.
Since $\bar u$ is supported on the finite set $F$,
the gradient of the convex function $\varphi$
only takes finitely many values
(outside a set of Lebesgue measure zero).
Consequently,
$\varphi$ is a finite supremum of affine functions
(with slopes $a_i$)
and this motivates the next definition.

\begin{definition}
	\label{def:varphi_xi}
	For a weight $\xi \in \R^n$, we define
	the function $\varphi_\xi \colon D \to \R$ via
	\[
		\varphi_\xi(x)
		:=
		\sup \set*{a_i^\top x - \xi_i \given i = 1, \ldots, n}
		.
	\]
	We define the \emph{dominating regions}
	\begin{equation*}
		D_i(\xi)
		:=
		\set*{x \in \interior(D) \given a_i^\top x - \xi_i > a_j^\top x - \xi_j \text{ for all } j \ne i}
		.
	\end{equation*}
\end{definition}
It is clear that $\varphi_\xi$ is differentiable on the open set $D_i(\xi)$
and we have
\begin{equation}
	\nabla \varphi_\xi(x) = a_i
	\quad\forall x \in D_i(\xi)
	.
\end{equation}
The set $D \setminus \bigcup_i D_i(\xi)$ is a Lebesgue null set and thus a $u_d$ null set.
Therefore, we can ignore these points
and use $\nabla\varphi_\xi$ as a transport map.
The measure $u_d$ is transported to
$\nabla \varphi_\xi \mathbin{\#} u_d = \sum_{i=1}^n u_d(D_i(\xi)) \delta_{a_i} \in \Uad$.
As above,
we identify this measure with the vector
\[
	\tilde u(\xi)
	=
	(u_d(D_1(\xi)), \ldots, u_d(D_n(\xi)))^\top
	\in \R^n
	.
\]
Since $\nabla \varphi_\xi$ is the transport map,
we have
\begin{equation}
	\label{eq:eval_transport_dist}
	W_2^2(u_d,\tilde u(\xi))
	=
	\int_D \abs{x - \nabla \varphi_\xi(x)}^2 \d u_d(x)
	.
\end{equation}
We mention that this formula also follows from
\cite[Corollary~1.2]{KitagawaMerigotThibert2019}
or
\cite[(3.3)--(3.5)]{LebratDeGournayKahnWeiss2019}.
At this point, we further mention that the only efficient possibility
for the evaluation of $W_2^2(u_d, u)$
(that we are aware of)
is to find a vector $\xi \in \R^n$ such that $u = \tilde u(\xi)$
and to compute the above integral.
This is the approach used in, e.g., \cite{KitagawaMerigotThibert2019}.

The above considerations lead to the problem
\begin{equation*}
	\tag{P\mbox{\(_\xi\)}}
	\label{P(xi)}
	\min_{\xi \in \R^n}
	j(\xi) + \frac \alpha 2 W(\xi)
\end{equation*}
where
$j(\xi) := g(\tilde u(\xi))$ and
$W(\xi) := W_2^2(u_d,\tilde u(\xi))$.

\begin{lemma}[Equivalence of \texorpdfstring{\eqref{P}}{(P)} and \texorpdfstring{\eqref{P(xi)}}{(Pxi)}]\leavevmode
	\label{lem:P and P(xi) equivalent}
	\begin{enumerate}
		\item
			\label{lem:P and P(xi) equivalent:i}
			Let $\xi \in \R^n$ be arbitrary. Then, $u := \nabla \varphi_\xi \mathbin{\#} u_d \in \Uad$.
		\item
			\label{lem:P and P(xi) equivalent:ii}
			Let $u \in \Uad$ be arbitrary. Then, there exists $\xi \in \R^n$
			such that $u = \nabla \varphi_\xi \mathbin{\#} u_d$.
	\end{enumerate}
	In both cases, the objective values of
	\eqref{P} and \eqref{P(xi)} coincide.
\end{lemma}
\begin{proof}
	\ref{lem:P and P(xi) equivalent:i}:
	This follows from the discussion above.

	\ref{lem:P and P(xi) equivalent:ii}:
	From \cite[Proposition~1.11]{Santambrogio}
	we get the existence of a solution $\psi \in C(F)$
	of the dual Kantorovich problem, i.e.,
	\begin{equation*}
		\int_D \psi^c(x_1) \d u_d(x_1)
		+
		\int_F \psi(x_2) \d u(x_2)
		=
		W_2^2(u_d, u)
		.
	\end{equation*}
	We define $\xi \in \R^n$ via
	\begin{equation}
		\label{eq:xi_from_psi}
		\xi_i := \frac12 \abs{a_i}^2 - \frac12 \psi(a_i)
		.
	\end{equation}
	This implies
	\begin{equation*}
		\varphi_\xi(x_1)
		=
		\sup\set*{
			x_1^\top a_i
			-
			\parens*{
				\frac12 \abs{a_i}^2
				-
				\frac12 \psi(a_i)
			}
			\given
			i = 1,\ldots, n
		}
		=
		\frac12 \abs{x_1}^2 - \frac12\psi^c(x_1)
		.
	\end{equation*}
	Let $\bar\gamma$ be the optimal transport plan
	from $u_d$ to $u$.
	From Kantorovich duality,
	we get that
	$\psi(x_2) + \psi^c(x_1) = c(x_1,x_2)$
	for all $(x_1,x_2) \in \supp(\bar\gamma)$.
	On the other hand, we have
	$\psi(x_2) + \psi^c(\hat x_1) \le c(\hat x_1,x_2)$
	for all $\hat x_1 \in \R^2$.
	By subtracting these two equations,
	we end up with
	\begin{equation*}
		\varphi_\xi(\hat x_1) \ge \varphi_\xi(x_1) + x_2^\top(\hat x_1 - x_1)
		\qquad
		\forall (x_1,x_2) \in \supp(\bar\gamma), \hat x_1 \in \R^2.
	\end{equation*}
	This shows $\supp(\bar\gamma) \subset \partial\varphi_\xi$.
	From the proof of Brenier's theorem \cite[Theorem~2.5.10]{FigalliGlaudo2023},
	it follows that $u = \nabla\varphi_\xi \mathbin{\#} u_d$.

	In both cases,
	the above construction ensures
	$u = \tilde u(\xi)$.
	This yields
	$j(\xi) = g(\tilde u(\xi)) = g(u)$
	and
	$W(\xi) = W_2^2(u_d, \tilde u(\xi)) = W_2^2(u_d, u)$,
	i.e., the objective values coincide.
\end{proof}

We mention that the vector $\xi$
constructed in
\ref{lem:P and P(xi) equivalent:ii}
is not unique.
First, we can add constants to $\xi$
without changing $\nabla\varphi_\xi$.
Moreover, if $D_i(\xi)$ is empty
for some index $i$,
we can change $\xi_i$ without affecting $\varphi_\xi$.

\begin{theorem}[Optimality condition]
	\label{lem:optimality of xi on J}
	For every $\bar u$ satisfying \eqref{eq:optimality},
	there exists a unique
	$\bar \xi \in \R^n$ with
	$\bar u = \tilde u (\bar \xi)$
	and
	$r(\bar \xi) = 0$,
	where
	\[
		r(\xi)
		:=
		A + \frac 1 \alpha \nabla g(\tilde u(\xi)) - \xi
		\qquad
		\text{and}
		\qquad
		A
		:=
		\frac 1 2
		\left(
			\abs{a_1}^2,
			\hdots,
			\abs{a_n}^2
		\right)^\top
		.
	\]
	Moreover, if $\bar \xi$ with $r(\bar \xi) = 0$
	is given,
	$\bar u := \tilde u(\bar \xi)$ satisfies \eqref{eq:optimality}.

	If $g$ is additionally convex, there exists a unique zero $\bar \xi \in \R^n$ of $r$ which is a minimizer
	of \eqref{P(xi)}.
\end{theorem}
\begin{proof}
	Let $\bar u \in \Uad$ satisfy \eqref{eq:optimality}.
	Then, $\psi := -2 g'(\bar u) / \alpha$
	is a solution of the Kantorovich dual problem.
	We define $\bar \xi$ as in \eqref{eq:xi_from_psi}
	and this yields
	$r(\bar \xi) = 0$.
	If we also have $\bar u = \tilde u(\xi)$ and $r(\xi) = 0$
	for some $\xi \in \R^n$,
	we get
	\begin{equation*}
		\bar \xi
		=
		A + \frac1\alpha \nabla g(\tilde u(\bar \xi))
		=
		A + \frac1\alpha \nabla g(\tilde u(\xi))
		=
		\xi
	\end{equation*}
	and this shows uniqueness.

	The uniqueness assertion in the convex case
	follows with \cref{thm:uniqueness}
	and \cref{lem:optimality condition of u on G}.
\end{proof}

Note that the approach of \cref{def:varphi_xi}
is related to the concept of power diagrams
(also known as Laguerre tessellations)
which is used in \cite{Merigot2011,KitagawaMerigotThibert2019}.
Indeed, we get
\begin{equation*}
	D_i(\xi)
	=
	\set*{x \in \interior(D) \given
		\frac12 \abs{x - a_i}^2 + \xi_i - \frac12 \abs{a_i}^2
		<
		\frac12 \abs{x - a_j}^2 + \xi_j - \frac12 \abs{a_j}^2
		\;\forall j \ne i
	}
	.
\end{equation*}

\subsection{Generalized concepts of differentiability}
\label{subsec:concepts_of_differentiability}

The functions $j$, $W$, $\tilde u$, and $r$,
which were defined using $\varphi_\xi$
are, in general,
not differentiable.
In the next section,
we will see that these functions are
amenable
to generalized concepts of differentiability.
In particular, we will use
the notion of
$PC^1$ functions.
We recall these concepts following \cite{Semismooth}.

\begin{definition}
	[\texorpdfstring{\cite[Def.~2.1]{Semismooth}}{}]
	\label{def:Clarke-Ableitung}
	Let $V \subset \R^n$ be open and $f: V \to \R^m$ be Lipschitz continuous near $x \in V$.
	Let $D_f := \set{x \in V \given f \text{ is differentiable at } x}$.
	The set
	\[
		\partial^C f(x)
		:=
		\conv(\set{M \in \R^{m \times n} \given \exists \seq{x_k} \subset D_f : x_k \to x, f'(x_k) \to M})
	\]
	is called Clarke's generalized Jacobian of $f$ at $x$.
\end{definition}
Note that
Clarke's generalized Jacobian is denoted by $\partial f$
in \cite{Semismooth}.
In order to avoid confusion with the convex subdifferential,
we denote it by $\partial^C f$.

\begin{definition}
	[\texorpdfstring{\cite[Def.~2.19]{Semismooth}}{}]
	A function $f : V \to \R^m$ defined on the open set $V \subset \R^n$ is called
	$PC^k$-function (``P'' for piecewise), $1 \le k \le \infty$, if f is continuous and if at every point
	$x_0 \in V$ there exist a neighborhood $U \subset V$ of $x_0$ and a finite collection of $C^k(U)$-functions
	$f^i : U \to \R^m$, $i = 1, \ldots, N$ , such that
	\[
		\forall x \in U:
		\qquad
		f(x)
		\in
		\set{f^1(x), \ldots, f^N(x)}
		.
	\]
	We say that $f$ is a continuous selection of $\set{f^1, \ldots, f^N}$ on $U$.
	The set
	$I(x) = \set{i \given f(x) = f^i(x)}$
	is the active set at $x \in U$, and
	\[
		I^e(x)
		=
		\set{i \in I(x) \given x \in \cl(\interior(\set{y \in U \given f(y) = f^i(y)}))}
	\]
	is the essentially active index set at $x$.
\end{definition}

In the next lemma, $f'(x; y)$ denotes the directional derivative of a function $f$
at the point $x$ in direction $y$,
i.e., $f'(x; y) = \lim_{t \searrow 0} (f(x + t y) - f(x))/t$.

\begin{lemma}
	[\texorpdfstring{\cite[Props.~2.24, 2.25]{Semismooth}}{}]
	\label{lem:PC1-functions:form of derivative}
	Let the $PC^1$-function $f: V \to \R^m$, $V \subset \R^n$ open, be a continuous selection of the $C^1$-functions $\set{f^1, \ldots, f^N} \in C^1(U)$ in a neighborhood $U$ of $x \in V$.
	Then
	$\partial^C f(x) = \conv \set{(f^i)'(x) \given i \in I^e(x)}$
	and
	for all $y \in \R^n$ we have
	$f'(x;y) \in \set{(f^i)'(x)y \given i \in I^e(x)}$.
	Further, if $f$ is differentiable at $x$, then
	$f'(x) \in \set{(f^i)'(x) \given i \in I^e(x)}$.
\end{lemma}

\begin{lemma}
	[\texorpdfstring{\cite[Def.~2.5, Props.~2.7 and 2.26]{Semismooth}}{}]
	\label{lem:PC1-functions are semismooth}
	Let $f: V \to \R^m$ be a $PC^1$-function on the open set $V \subset \R^n$.
	Then $f$ is semismooth, i.e.,
	$f$ is Lipschitz continuous near $x$, $f'(x;\cdot)$ exists and
	\[
		\sup_{M \in \partial^C f(x+s)} \abs{f(x+s) - f(x) - Ms}
		=
		\smallO(\abs{s})
		\quad
		\text{as } s \to 0.
	\]
\end{lemma}

\section{Properties of \texorpdfstring{\eqref{P(xi)}}{(Pxi)}}
\label{sec:properties_Pxi}
In this section,
we study properties of the objective function
\begin{equation*}
	J(\xi)
	:=
	j(\xi) + \frac \alpha 2 W(\xi)
\end{equation*}
of problem \eqref{P(xi)}
and of the map $\xi \mapsto \tilde u(\xi)$.
We will see that these functions are differentiable
for almost all $\xi \in \R^n$
and this will be used to show their semismoothness.

\subsection{Differentiability condition and sectors}

We will introduce a criterion for differentiability first.

\begin{definition}
	\label{def:neu}
	Let $k \in \N$ be as small as possible, $\nu_{-1}, \ldots, \nu_{-k} \in \R^2$,
	and $\zeta_{-1}, \ldots, \zeta_{-k} \in \R$ be given
	such that $D = \bigcap_{i=-k}^{-1} \set{x \in \R^2 \given \nu_i^\top x - \zeta_i \le 0}$.
	By $e_{-i} := \set{x \in D \given \nu_{-i}^\top x = \zeta_{-i}}$
	we denote the edges of $D$.
	With this, we define the \emph{active indices}
	\[
		I_\xi(x)
		=
		\set*{i \in \set{1,\ldots,n} \given \varphi_\xi(x) = a_i^\top x - \xi_i}
		\cup
		\set*{-i \in \set{-1, \ldots, -k} \given x \in e_{-i}}
		.
	\]
	We say that
	a weight vector $\xi \in \R^n$ fulfills the
	\emph{differentiability condition},
	if
	\begin{align}
		\label{DC}
		\tag{DC}
		\begin{split}
			\forall q \in D: \abs{I_\xi(q)} \le 3
			.
		\end{split}
	\end{align}
	For $\xi \in \R^n$ satisfying \eqref{DC}
	we define the total configuration
	\[
		k(\xi) := \set*{I_\xi(x) \given x \in D, \abs{I_\xi(x)} = 3}
		.
	\]
	Since there are only finitely many possibilities
	for the total configuration
	$k(\xi)$,
	the set of weights fulfilling \eqref{DC}
	can be partitioned into \emph{sectors}
	$\Xi^1, \ldots, \Xi^N$
	on which $k(\cdot)$ is constant,
	i.e., the sectors are the equivalence classes w.r.t.\ the equivalence relation
	$\xi \sim \hat\xi$ if and only if $k(\xi) = k(\hat\xi)$.
	For convenience, we set $k(\Xi^i) = k(\xi)$, where $\xi \in \Xi^i$ is arbitrary.
\end{definition}

Note that if $\xi$ satisfies \eqref{DC}, $k(\xi)$ is always nonempty:
Indeed for each vertex $v$ of $D$, there exist $i,j \in \set{1,\ldots,k}$, $i \ne j$,
such that $v = e_{-i} \cap e_{-j}$,
i.e., $-i, -j \in I_\xi(v)$.
Further, there has to be at least one positive entry $l \in \set{1,\ldots,n}$ with $l \in I_{\xi}(v)$
due to the definition of $\varphi_\xi$.
Since $\xi$ is assumed to satisfy \eqref{DC},
this $l$ has to be unique, $I_\xi(v) = \set{l, -i, -j}$
and, therefore, $\set{l, -i, -j} \in k(\xi)$.

\begin{remark}
	\label{rem:ein_remark}
	We define additional outer regions $D_{-1}, \ldots, D_{-k}$
	where $D_{-i}$ is the set of all points in $\R^2 \setminus D$ whose projection onto $D$ lies in the relative interior of $e_{-i}$.
	Basically, \eqref{DC} means that nowhere more than three regions (including these outer regions) touch.
	We will see later that \eqref{DC} is sufficient but not necessary for differentiability of $J$,
	see \cref{rem:differentiability}.
	For convenience, we sometimes use $D_{-i}(\xi) := D_{-i}$.
\end{remark}
Now, we define a point $q_{ijl}$ which is a candidate touching point for the regions $D_i$, $D_j$, and $D_l$.
\begin{definition}
	\label{def:q_ijl}
	\begin{subequations}
		For mutually distinct $i,j,l \in \set{1,\ldots,n}$, we define (if the matrix is invertible)
		\begin{equation}
			\label{eq:q_ijl(xi) all positive}
			q_{ijl}(\xi)
			:=
			(a_i - a_j, a_i - a_l)^{-\top}
			\begin{pmatrix}
				\xi_i - \xi_j \\
				\xi_i - \xi_l
			\end{pmatrix}
			.
		\end{equation}
		For distinct $i,j \in \set{1, \ldots, n}$ and $l \in \set{-1, \ldots, -k}$, we define (if the matrix is invertible)
		\begin{equation}
			\label{eq:q_ijl(xi) i and j positive}
			q_{ijl}(\xi)
			:=
			(a_i - a_j, \nu_l)^{-\top}
			\begin{pmatrix}
				\xi_i - \xi_j \\
				\zeta_l
			\end{pmatrix}
			.
		\end{equation}
		For $i \in \set{1, \ldots, n}$ and distinct $j,l \in \set{-1, \ldots, -k}$, we define (if the matrix is invertible)
		\begin{equation}
			\label{eq:q_ijl(xi) j l negative}
			q_{ijl}(\xi)
			:=
			(\nu_j, \nu_l)^{-\top}
			\begin{pmatrix}
				\zeta_j \\
				\zeta_l
			\end{pmatrix}
			.
		\end{equation}
	\end{subequations}
	By permutation,
	\eqref{eq:q_ijl(xi) i and j positive}
	and
	\eqref{eq:q_ijl(xi) j l negative}
	are extended to other combinations
	of positive and negative indices.
\end{definition}

\begin{lemma}
	\label{lem:vertices affinely dependent}
	Let $\xi \in \R^n$ with \eqref{DC} be given.
	Let $\set{i,j,l} \in k(\xi)$. Then, the function $q_{ijl}$ is defined.
	Moreover, $x = q_{ijl}(\xi)$ is the unique point in $D$ with $I_\xi(x) = \set{i,j,l}$.
	It exists $\varepsilon > 0$ such that for all $\tilde \xi \in U_\varepsilon(\xi)$ it holds $I_{\tilde \xi}(q_{ijl}(\tilde \xi)) = \set{i,j,l}$.
\end{lemma}
\begin{proof}
	We only consider the case $i,j,l > 0$.
	We denote by $M$ the set of all points
	$x \in D$ satisfying $I_\xi(x) = \set{i,j,l}$.
	For all $x \in M$, we have
	$a_i^\top x - \xi_i = a_j^\top x - \xi_j = a_l^\top x - \xi_l$.
	Assume that the matrix in \eqref{eq:q_ijl(xi) all positive}
	is not invertible, i.e., $a_i, a_j, a_l$ are not affinely independent.
	Then one of these points is a convex combination of the others, w.l.o.g.\ we have
	$a_j = \lambda a_i + (1-\lambda) a_l$ for some $\lambda \in (0,1)$.
	The set $M$ is a convex subset of an affine $1$-dimensional subspace.
	In fact, one can check that $M$ is the edge between $D_i(\xi)$ and $D_l(\xi)$ by considering points of the form $x + t (a_i - a_l)$ for $x \in M$ and $t \in \R$.
	The solution set has to be an edge which ends in a vertex. But in this vertex, another (fourth) index is active. This violates \eqref{DC}.
	Thus, $a_i, a_j, a_l$ are affinely independent, the matrix in \eqref{eq:q_ijl(xi) all positive} is invertible and $q_{ijl}(\xi)$
	is the only point satisfying
	$a_i^\top x - \xi_i = a_j^\top x - \xi_j = a_l^\top x - \xi_l$,
	i.e., $M = \set{q_{ijl}(\xi)}$.

	All other indices $r$ are not active in $q_{ijl}(\xi)$,
	i.e., $a_r^\top q_{ijl}(\xi) < a_i^\top q_{ijl}(\xi)$
	and, by continuity,
	these inequalities hold for all $\tilde \xi$
	in a neighborhood of $\xi$.

	The cases in which one of the indices $i,j,l$ is negative
	can be discussed analogously.
\end{proof}

\begin{theorem}
	\label{thm:sectors are convex and open}
	The sectors $\Xi^r$, $r = 1,\ldots,N$, are convex and open.
\end{theorem}
\begin{proof}
	Let $\mathbb{I} := k(\Xi^r)$.
	We check that the sector $\Xi^r$ is the set of all solutions $\xi \in \R^n$ of
	\begin{align*}
		&\forall \set{i,j,l} \in \mathbb{I}, i > 0:
		\forall s > 0, s \notin \set{i,j,l}:
		&
		a_s^\top q_{ijl}(\xi) - \xi_s &< a_i^\top q_{ijl}(\xi) - \xi_i
		, \\
		&\forall \set{i,j,l} \in \mathbb{I}: \phantom{i > 0,{}}
		\forall s < 0, s \notin \set{i,j,l}:
		&
		\nu_s^\top q_{ijl}(\xi) - \zeta_s &< 0
		.
	\end{align*}
	If this is verified, we easily get that $\Xi^r$ is convex and open.

	Let $\xi \in \Xi^r$.
	\Cref{lem:vertices affinely dependent} yields that no other index is active in $q_{ijl}(\xi)$,
	hence the two conditions follow.

	Let, on the other hand, $\xi \in \R^n$ be a solution of the above system.
	We have to check that \eqref{DC} holds at $\xi$ and $\xi \in \Xi^r$, i.e., $k(\xi) = \mathbb{I}$.
	For any triple $\set{i, j, l} \in \mathbb{I}$, the
	definition of the points $q_{ijl}(\xi)$
	and the validity of the above system
	readily imply
	$q_{ijl}(\xi) \in D$
	and $I_\xi(q_{ijl}(\xi)) = \set{i, j, l}$.
	If we would already know that \eqref{DC} is satisfied in $\xi$,
	this shows $\set{i, j, l} \in k(\xi)$ for all $\set{i, j, l} \in \mathbb{I}$,
	i.e.,
	$\mathbb{I} \subset k(\xi)$.
	It remains to check that $\xi$ satisfies \eqref{DC} and $\mathbb{I} = k(\xi)$.

	To this end,
	we show that the structure of the regions $D_i(\xi)$ can already be inferred from $\mathbb{I}$.
	We encode this structure in
	a graph $G_\xi = (V_\xi, E_\xi)$
	using the regions $D_i(\xi)$ from \cref{def:varphi_xi}
	and the outer regions $D_{-i}$ from \cref{rem:ein_remark}.
	We set
	\begin{equation*}
		V_\xi
		:=
		\set{
			I \subset \set{1,\ldots,n}\cup\set{-1,\ldots,-k}
			\given
			\abs{I} \ge 3,
			\exists x \in D:
			\forall i: i \in I \Leftrightarrow x \in \cl(D_i(\xi))
		},
	\end{equation*}
	i.e., the vertices in $V_\xi$
	correspond to the vertices of the regions $D_i(\xi)$.
	The edge set encodes the edges of the regions, i.e.,
	\begin{equation*}
		E_\xi
		:=
		\set{
			\set{I_1, I_2}
			\given
			I_1, I_2 \in V_\xi,
			\abs{ I_1 \cap I_2 } = 2
		}
		.
	\end{equation*}
	This is a connected graph.
	Similarly, we define a graph
	$G_{\mathbb{I}} = (V_{\mathbb{I}}, E_{\mathbb{I}})$
	by using some $\xi_r \in \Xi^r$, i.e., $k(\xi_r) = \mathbb{I}$.
	Since \eqref{DC} is satisfied for $\xi_r$,
	we can check $V_{\mathbb{I}} = \mathbb{I}$.
	We will show that both graphs coincide.

	Given a triple $\set{i, j, l} \in \mathbb{I}$,
	the point $q_{ijl}(\xi)$ ensures $\set{i, j, l} \in V_\xi$,
	since the linear inequality system shows that
	$q_{ijl}(\xi)$ does not belong to the closure of $D_m(\xi)$
	for all $m \not\in \set{i, j, l}$.
	Thus, $V_{\mathbb{I}} \subset V_\xi$.
	Similarly, if
	$\set{I_1, I_2} \in E_{\mathbb{I}}$,
	we have $\abs{I_1 \cap I_2} = 2$
	and $I_1, I_2 \in V_{\mathbb{I}} \subset V_\xi$.
	Consequently, $G_{\mathbb{I}}$
	is a subgraph of $G_\xi$.

	In order to check equality of both graphs,
	we note that any vertex
	$\set{i, j, l} \in \mathbb{I}$
	has degree $2$
	(if exactly two of $i, j, l$ are negative)
	or $3$
	(if at most one of $i, j, l$ is negative)
	in both graphs $G_\xi$ and $G_{\mathbb{I}}$.
	Hence, every edge from $E_\xi$ incident to $\set{i, j, l}$
	is already present in $E_{\mathbb{I}}$.
	This shows that both graphs coincide.

	Let us check that \eqref{DC} is satisfied at $\xi$.
	By definition of the graphs and $V_\xi = V_{\mathbb{I}}$,
	for every $x \in D$ with $x \in \cl(D_i(\xi))$
	for at least three indices $i$,
	we have $x = q_{ijl}(\xi)$ for some $\set{i,j,l} \in \mathbb{I}$.
	Thus, $I_\xi(x) = \set{i,j,l}$ follows from the inequality system.
	On the other hand, if $x$ lies on the edge between
	$q_{ijl}(\xi)$ and $q_{ijm}(\xi)$,
	i.e., $x = \lambda q_{ijl}(\xi) + (1 - \lambda) q_{ijm}(\xi)$
	for $\lambda \in (0,1)$,
	$I_\xi(x) = \set{i,j}$ follows from the inequality system.
	All other $x \in D$ belong to a region $D_i(\xi)$
	and, thus, $I_\xi(x) = \set{i}$.
	This shows \eqref{DC}.

	The definition of $k(\xi)$ now yields
	$k(\xi) = V_\xi = V_{\mathbb{I}} = \mathbb{I}$,
	hence $\xi \in \Xi^r$.
	This finishes the proof.
\end{proof}
\begin{remark}[Reconstruction of regions]
	\label{remark:reconstruction of regions from sector}
	The proof of \cref{thm:sectors are convex and open}
	actually shows
	that
	if $\xi$ satisfies \eqref{DC},
	we can reconstruct the whole partition of $D$ into
	the regions $D_i(\xi)$
	by looking at the configuration $k(\xi)$.
	In particular, two regions $D_i(\xi)$ and $D_j(\xi)$
	are adjacent if and only if
	$\set{i, j, l} \in k(\xi)$
	for some index $l$.
	Note that there exist exactly two indices $l$ and $m$
	with this property
	and the points
	$q_{i j l}(\xi)$ and $q_{i j m}(\xi)$
	are the vertices of the edge between
	$D_i(\xi)$ and $D_j(\xi)$.
\end{remark}

Now we show that \eqref{DC}
is valid at almost all points.
\begin{theorem}
	\label{thm:degenerated sectors are zero-set}
	The set
	$S := \set{\xi \in \R^n \given \eqref{DC} \text{ is violated}}$
	is a $\lambda^n$ null set.
	In particular, $\bigcup_{i=1}^N \cl(\Xi^i) = \R^n$ holds.
\end{theorem}
\begin{proof}
	By definition,
	$S = \set*{\xi \in \R^n \given \exists q \in D : \abs{I_\xi(q)} \ge 4}$.
	This is the union of the sets
	$M_{ijlm} := \set{\xi \in \R^n \given \exists q \in D: i,j,l,m \in I_\xi(q)}$
	for all pairwise distinct $i,j,l,m \in \set{-k, \ldots,n} \setminus \set{0}$.
	It is sufficient to show that these sets have Lebesgue measure zero.
	We consider the case that all indices are positive,
	the other cases being similar.
	We define the matrix
	$A = (a_i - a_j, a_i - a_l, a_i - a_m)^\top \in \R^{3 \times 2}$
	and the matrix $T \in \R^{3 \times n}$ via
	$T \xi := (\xi_i - \xi_j, \xi_i - \xi_l, \xi_i - \xi_m)^\top$.
	We have
	\[
		M_{ijlm}
		\subset
		\set*{
			\xi \in \R^n \given
			A q = T \xi
			\text{ has a solution }
			q \in \R^2
		}
		=
		T^{-1}( A \R^2 )
		.
	\]
	Note that this is only an inclusion,
	since we no longer require $q \in D$ in the right-hand side.
	The matrix $A$ has rank at most $2$,
	thus the image $A \R^2$ is at most two dimensional.
	Further, $\dim(\ker(T)) = n - 3$.
	Combined, $\dim(M_{ijlm}) \le n - 3 + 2 \le n - 1$ and thus $\lambda^n(M_{ijlm}) = 0$.
\end{proof}

\subsection{Continuity and differentiability of \texorpdfstring{$J$}{J}}
\label{subsec:diff}
Now, we are going to
prove properties of the functions $J, W, j, \tilde u$
appearing in the problem \eqref{P(xi)}.
For convenience, we recall the definition of $W$,
\begin{equation}
	\label{eq:W(xi) as sum of integrals over regions}
	W(\xi)
	:=
	\int_D \abs{x - \nabla \varphi_\xi(x)}^2 \d u_d(x)
	=
	\sum_{i=1}^n \int_{D_i(\xi)} \abs{x}^2 - 2 a_i^\top x + \abs{a_i}^2 \d u_d(x)
	.
\end{equation}

\begin{theorem}
	\label{thm:J_r_W_tilde_u_Lipschitz}
	The maps $\tilde u, r, W, J$ are Lipschitz continuous.
\end{theorem}
\begin{proof}
	We start with the components of the function $\tilde u \colon \R^n \to \R^n$.
	Let $\xi, \hat \xi \in \R^n$ be given.
	The density function $\varrho$ of $u_d$
	is continuous and
	hence bounded on $D$ by some constant $K > 0$.
	This yields
	\begin{equation}
		\label{eq:first_estimate}
		\begin{aligned}
			\abs{\tilde u_i(\xi) - \tilde u_i(\hat \xi)}
			&=
			\abs{u_d(D_i(\xi)) - u_d(D_i(\hat \xi))}
			\\&
			\le
			u_d(D_i(\xi) \setminus D_i(\hat \xi)) + u_d(D_i(\hat \xi) \setminus D_i(\xi))
			\\
			&\le
			K \left( \lambda^2(D_i(\xi) \setminus D_i(\hat \xi)) + \lambda^2(D_i(\hat \xi) \setminus D_i(\xi)) \right)
			.
		\end{aligned}
	\end{equation}
	The measure of the first set difference can be bounded by
	\begin{equation*}
		\lambda^2(D_i(\xi) \setminus D_i(\hat \xi))
		\le
		\sum_{j \ne i} \lambda^2\parens{D \cap \set{x \in \R^2 \given \hat \xi_i - \hat \xi_j \ge (a_i - a_j)^\top x > \xi_i - \xi_j}}
		.
	\end{equation*}
	For every $j$ the set on the right-hand side
	is a strip of width
	$\frac{\abs{\xi_i - \xi_j - \hat \xi_i + \hat \xi_j}}{\abs{a_i - a_j}}$
	intersected with $D$.
	Since the diameter $\hat L$ of $D$ is finite,
	we get
	\begin{equation}
		\label{eq:second_estimate}
		\abs{\tilde u_i(\xi) - \tilde u_i(\hat \xi)}
		\le
		2K
		\hat L
		\sum_{j \ne i} \frac{\abs{\xi_i - \xi_j - \hat \xi_i + \hat \xi_j}}{\abs{a_i - a_j}}
		\le
		4 K \hat L \sum_{j \ne i} \frac 1 {\abs{a_i - a_j}}
		\abs{\xi - \hat \xi}
		.
	\end{equation}
	This shows that $\tilde u$ is Lipschitz continuous.

	Now we consider the function $j = g \circ \tilde u$.
	The function $g$ is Lipschitz continuous
	on the compact set $\Uad$,
	since it is continuously Fréchet differentiable.
	Since $\tilde u$ maps to $\Uad$,
	$j$ is the composition of Lipschitz continuous functions and therefore Lipschitz.

	We continue with $W$ using \eqref{eq:W(xi) as sum of integrals over regions}.
	The integrand $\abs{x-a_i}^2$
	is bounded by some constant $\kappa > 0$ as $D, F$ are compact.
	We get
	\begin{align*}
		\abs{W(\xi) - W(\hat \xi)}
		&\le
		\sum_{i=1}^n
		\left[
			\int_{D_i(\xi) \setminus D_i(\hat \xi)} \abs{x - a_i}^2 \d u_d(x)
			+
			\int_{D_i(\hat \xi) \setminus D_i(\xi)} \abs{x - a_i}^2 \d u_d(x)
		\right]
		\\
		&\le
		\kappa
		\sum_{i=1}^n
		\left( u_d(D_i(\xi) \setminus D_i(\hat \xi)) + u_d(D_i(\hat \xi) \setminus D_i(\xi)) \right)
		.
	\end{align*}
	Arguing as in \eqref{eq:first_estimate}, \eqref{eq:second_estimate}
	yields the Lipschitz continuity of $W$.
	Finally,
	$J = j + \frac \alpha 2 W$ is Lipschitz as well.

	As for $r$, it is sufficient to check that $\xi \mapsto \nabla g(\tilde u(\xi))$
	is Lipschitz continuous.
	The function $\nabla g$ is Lipschitz continuous on
	the compact set $\Uad$, since $g$ is $C^2$.
	Further, $\tilde u$ is Lipschitz continuous as proven above.
	This shows the claim.
\end{proof}
For each $\xi$,
we define a matrix which will serve as
(a substitute of) the derivative of $\tilde u$.
\begin{definition}
	\label{def:theta}
	Given $\xi \in \R^n$,
	we define the matrix $\Theta(\xi) \in \R^{n \times n}$.
	Let $i \in \set{1,\ldots,n}$ be given.

	Case 1: $u_d(D_i(\xi)) > 0$.
	Let $i_1, \ldots, i_m$ be the indices of the neighboring regions
	(sharing an edge with $D_i(\xi)$)
	in mathematically positive order.
	For convenience, we set
	$i_0 := i_m$ and $i_{m+1} := i_1$.
	Note that $q_{i i_{s-1} i_s}(\xi)$
	is well defined and is the vertex between
	$D_i(\xi)$, $D_{i_{s-1}}(\xi)$, $D_{i_s}(\xi)$;
	and the edge between $D_i(\xi)$ and $D_{i_s}(\xi)$
	is the segment from
	$q_{i i_{s-1} i_s}(\xi)$
	to
	$q_{i i_s i_{s+1}}(\xi)$.
	We set
	\[
		\Theta_{ik}(\xi)
		:=
		\begin{cases}
			\displaystyle
			- \sum_{s : i_s > 0}
			\frac{1}{\abs{a_{i_s} - a_i}}
			\int_{q_{i i_{s-1} i_s}(\xi)}^{q_{i i_s i_{s+1}}(\xi)} \varrho(x) \d \HH^1(x)
			&\text{if } k = i,
			\\
			\displaystyle
			\frac{1}{\abs{a_{i_s} - a_i}}
			\int_{q_{i i_{s-1} i_s}(\xi)}^{q_{i i_s i_{s+1}}(\xi)} \varrho(x) \d \HH^1(x)
			&\text{if } k = i_s \text{ for some } s \ge 1,
			\\
			0
			&\text{otherwise.}
		\end{cases}
	\]
	Here, $\HH^1$ is the one-dimensional Hausdorff measure.

	Case 2: $u_d(D_i(\xi)) = 0$.
	We set
	$\Theta_{ij}(\xi) := 0$ for all $j \in \set{1,\ldots,n}$.
\end{definition}
Roughly speaking,
an off-diagonal entry $\Theta_{i k}(\xi)$, $i \ne k$
is non-zero only if $D_i(\xi)$ and $D_k(\xi)$
are adjacent and the entry is
a weighted length of the corresponding edge.
The diagonal entries are chosen
such that the row-wise sum is zero.

\begin{lemma}
	\label{lem:properties of Theta}
	The function $\Theta \colon \R^n \to \R^{n \times n}$
	has the following properties.
	\begin{enumerate}
		\item
			\label{lem:properties of Theta:i}
			For all $\xi \in \R^n$,
			the matrix $\Theta(\xi)$ is symmetric and negative semidefinite.
		\item
			\label{lem:properties of Theta:ii}
			There exists a constant $C > 0$, such that for all $\xi \in \R^n$,
			$\norm{\Theta(\xi)} \le C$ holds.
		\item
			\label{lem:properties of Theta:iii}
			For every $k \in \set{1, \ldots, N}$,
			$\Theta$ is uniformly continuous on the sector $\Xi^k$
			(see \cref{def:neu}).
	\end{enumerate}
\end{lemma}
In \ref{lem:properties of Theta:ii}, $\norm{\cdot}$ denotes the spectral norm of a matrix.
\begin{proof}
	\ref{lem:properties of Theta:i}:
	We denote by
	\begin{equation*}
		E(\xi) :=
		\set{
			\set{i, j} \subset \set{1,\ldots,n}
			\given
			\text{$D_i(\xi)$ and $D_j(\xi)$ share an edge}
		}
	\end{equation*}
	the indices corresponding to the inner edges
	(of the partition of $D$ into the regions $D_i(\xi)$).
	For $\set{i, j} \in E(\xi)$,
	we set
	$Q_{ij}(\xi) := \int_{p}^{q} \varrho(x) \d \HH^1(x) / \abs{a_{j} - a_i} \ge 0$,
	where $p$ and $q$ are the vertices of the edge between
	$D_i(\xi)$ and $D_j(\xi)$.
	We further
	define the matrix $A^{ij}(\xi) \in \R^{n \times n}$ via
	\[
		A^{ij}(\xi)
		:=
		-Q_{ij}(\xi)
		(e_i - e_j) (e_i - e_j)^\top
		,
	\]
	where $e_i, e_j$ are unit vectors in $\R^n$.
	It is easy to check
	$\Theta(\xi) = \sum_{\set{i,j} \in E(\xi)} A^{ij}(\xi)$,
	i.e., $\Theta(\xi)$ is the sum of symmetric negative semidefinite matrices.

	\ref{lem:properties of Theta:ii}:
	As $\Theta(\xi)$ is the sum of at most $n(n-1)/2$ matrices of type $A^{ij}$, it is sufficient to check their boundedness.
	The numbers $Q_{ij}(\xi)$ are uniformly bounded as the density $\varrho$ is bounded on the compact set $D$, hence the weighted length of the edge between two vertices is bounded.

	\ref{lem:properties of Theta:iii}:
	Let $\xi, \tilde \xi \in \Xi^k$ be arbitrary.
	\Cref{remark:reconstruction of regions from sector} allows us to
	reconstruct the structure of the regions from the configuration.
	In particular, we have $E(\xi) = E(\tilde \xi)$.
	Thus,
	\[
		\norm{\Theta(\xi) - \Theta(\tilde \xi)}
		\le
		\sum_{\set{i,j} \in E(\xi)} \norm*{A^{ij}(\xi) - A^{ij}(\tilde \xi)}
		=
		\sqrt{2} \sum_{\set{i,j} \in E(\xi)} \abs*{Q_{ij}(\xi) - Q_{ij}(\tilde \xi)}
		.
	\]
	Here, $\sqrt{2}$ is the norm of the matrix
	$(e_i - e_j) (e_i - e_j)^\top$.
	In order to estimate $Q_{ij}$,
	let $l \ne m$ be those indices with
	$\set{i,j,l}, \set{i,j,m} \in k(\Xi^k)$,
	see \cref{remark:reconstruction of regions from sector}.
	Now, $Q_{ij}$ is the composition of
	\begin{align*}
		\Xi^k \ni \xi
		&\mapsto
		(q_{ijl}(\xi), q_{ijm}(\xi))
		\in D \times D
		,
		\\
		D \times D \in (p,q)
		&\mapsto
		\int_p^q \varrho \d \HH^1(x) / \abs{a_j - a_i}
		=
		\int_0^1 \varrho(p+t(q-p)) \dt \abs{p-q} / \abs{a_j - a_i}
		\in \R
		.
	\end{align*}
	The former function is affine, cf.\ \cref{lem:vertices affinely dependent}, thus Lipschitz.
	The latter function is continuous,
	thus uniformly continuous (Heine--Cantor theorem).
	In total, $Q_{ij}$ is uniformly continuous
	on $\Xi^k$ for all $\set{i,j} \in E(\xi)$.
	In turn, $\Theta$, which is the sum of uniformly continuous functions, has this property as well.
\end{proof}

\begin{theorem}
	\label{thm:J_r_tilde u_W_differentiable_for_DC}
	Let $\xi \in \R^n$ with \eqref{DC}.
	Then, $\tilde u, r, W, J$ are continuously differentiable at $\xi$ with the derivatives
	$\tilde u'(\xi) = \Theta(\xi)$,
	$r'(\xi) = \frac 1 \alpha \nabla^2 g(\tilde u(\xi)) \Theta(\xi) - I$,
	$W'(\xi) = 2(A-\xi)^\top \Theta(\xi)$ and
	$J'(\xi) = \alpha r(\xi)^\top \Theta(\xi)$,
	where $A$ and $r$ were defined in \cref{lem:optimality of xi on J}.
\end{theorem}
\begin{proof}
	As \eqref{DC} holds and sectors are open,
	cf.\ \cref{thm:sectors are convex and open},
	the structure of the regions, i.e., the total configuration, stays constant for perturbations $\delta \xi \in \R^2$ with small enough norm $\abs{\delta\xi}$.
	We consider the directional derivative of $\tilde u$ in a direction $\delta \xi$. For every $\set{i, j, l} \in k(\xi)$,
	the map $q_{i j l}$ is affine, see \cref{lem:vertices affinely dependent}.
	We denote by
	$\delta q_{i j l}$
	the directional derivative of $q_{i j l}$
	at $\xi$ in direction $\delta \xi$,
	i.e., $\delta q_{i j l} := q_{i j l}'(\xi; \delta \xi)$.
	Next, we extend these perturbations
	of the vertices to the edges.
	That is,
	we define
	\begin{equation*}
		\tilde V
		\colon
		\set[\big]{
			\lambda q_{i j l}(\xi) + (1 - \lambda) q_{i j m}(\xi)
			\given
			\set{i, j, l}, \set{i, j, m} \in k(\xi),
			l \ne m,
			\lambda \in [0,1]
		}
		\to
		\R^2
	\end{equation*}
	via
	\begin{equation*}
		\tilde V\parens[\big]{
			\lambda q_{i j l}(\xi) + (1 - \lambda) q_{i j m}(\xi)
		}
		=
		\lambda \delta q_{i j l} + (1 - \lambda) \delta q_{i j m}
		.
	\end{equation*}
	The map $\tilde V$ is Lipschitz,
	hence, by the Kirszbraun theorem, there is a Lipschitz extension $V : \R^2 \to \R^2$.
	By construction of $V$,
	we have
	\begin{equation*}
		D_i(\xi + t \delta \xi)
		=
		(I + t V)(D_i(\xi))
	\end{equation*}
	for all $t \ge 0$ small enough,
	since we assumed that $\delta \xi$ does not change the total configuration.
	Applying \cite[Theorem~5.2.2]{HenrotPierro2018}
	to $\Phi(t) = I + t V$
	yields
	\begin{align*}
		\tilde u_i'(\xi;\delta \xi)
		:=
		\left( \frac{\d}{\d t} \int_{D_i(\xi + t \delta \xi)} \varrho(x) \dx \right) \Bigg|_{t=0}
		=
		0
		+
		\int_{\partial D_i(\xi)} \varrho(x) n(x)^\top V(x) \d \HH^1(x)
		.
	\end{align*}
	Here, $n$ is the outer unit normal vector.
	The boundary $\partial D_i(\xi)$ consists
	of outer edges (subsets of $\partial D$)
	and inner edges.
	On the outer edges, we have $n(x)^\top V(x) = 0$,
	since \eqref{DC} ensures that
	perturbations of boundary vertices stay on the boundary,
	see \cref{lem:vertices affinely dependent}.
	For an inner edge
	we have indices
	$j, l, m$
	such that
	the points $q_{i j l}(\xi)$ and $q_{i j m}(\xi)$ are the vertices,
	i.e., $D_i(\xi)$ and $D_j(\xi)$ are incident with the edge.
	Consequently, the outer unit normal vector on this edge is
	$(a_{j} - a_i)/(\abs{a_{j} - a_i})$.
	Further, on the original edge
	we have
	\begin{equation*}
		(a_{j} - a_i)^\top
		\parens{
			\lambda q_{i j l}(\xi) + (1 - \lambda) q_{i j m}(\xi)
		}
		=
		\xi_{j} - \xi_i
		\qquad\forall \lambda \in [0,1]
		,
	\end{equation*}
	while the perturbed points satisfy
	\begin{equation*}
		(a_{j} - a_i)^\top
		\parens{
			\lambda (q_{i j l}(\xi) + \delta q_{i j l}) + (1 - \lambda) (q_{i j m}(\xi) + \delta q_{i j m})
		}
		=
		(\xi_{j} + \delta \xi_j) - (\xi_i + \delta \xi_i)
	\end{equation*}
	for all $\lambda \in [0,1]$.
	Taking the difference yields
	\begin{equation*}
		(a_{j} - a_i)^\top
		V\parens[\big]{
			\lambda q_{i j l}(\xi) + (1 - \lambda) q_{i j m}(\xi)
		}
		=
		(a_{j} - a_i)^\top
		\parens{
			\lambda \delta q_{i j l} + (1 - \lambda) \delta q_{i j m}
		}
		=
		\delta \xi_j - \delta \xi_i
	\end{equation*}
	for all $\lambda \in [0,1]$.
	Thus,
	\begin{align*}
		\tilde u_i'(\xi;\delta \xi)
		&=
		\sum_{s : i_s > 0}
		\frac{\delta \xi_{i_s} - \delta \xi_i}{\abs{a_{i_s} - a_i}}
		\int_{q_{i i_{s-1} i_s}(\xi)}^{q_{i i_s i_{s+1}}(\xi)}
		\varrho(x)
		\d \HH^1(x)
		=
		\Theta_i(\xi) \delta \xi
		,
	\end{align*}
	where we used the same notation as in \cref{def:theta}.
	This yields $\tilde u'(\xi; \delta \xi) = \Theta(\xi) \delta \xi$.
	In particular, $\tilde u$ has continuous partial derivatives,
	see \itemref{lem:properties of Theta:iii},
	hence $\tilde u$ is continuously differentiable.

	We continue with $W$.
	Considering \eqref{eq:W(xi) as sum of integrals over regions},
	the integral over $\abs{x}^2$ is constant as
	we have
	$\sum_{i=1}^n \int_{D_i(\xi)} \abs{x}^2 \dx = \int_D \abs{x}^2 \dx$.
	For the other addends,
	we argue as above and
	apply \cite[Theorem~5.2.2]{HenrotPierro2018} again.
	We obtain
	\begin{align*}
		W'(\xi;\delta \xi)
		&=
		\sum_{i=1}^n
		\sum_{s : i_s > 0}
		\frac{\delta \xi_{i_s} - \delta \xi_i}{\abs{a_{i_s} - a_i}}
		\int_{q_{i i_{s-1} i_s}(\xi)}^{q_{i i_s i_{s+1}}(\xi)}
		\varrho(x)
		(\abs{a_i}^2 - 2 a_i^\top x)
		\d \HH^1(x)
		.
	\end{align*}
	For the term involving $\abs{a_i}^2$,
	we can reuse the computation from above.
	For the other term,
	we use the set of inner edges $E(\xi)$
	from the proof of \cref{lem:properties of Theta}.
	In the above sum,
	an edge $(i,j) \in E(\xi)$ is considered twice
	and the factor $\delta \xi_{i_s} - \delta \xi_i$
	changes its sign.
	This yields the representation
	\begin{equation*}
		W'(\xi;\delta \xi)
		=
		\sum_{i=1}^n
		\abs{a_i}^2 \Theta_i(\xi) \delta \xi
		-2 \!
		\sum_{\set{i,j} \in E(\xi)}
		\frac{\delta \xi_{j} - \delta \xi_i}{\abs{a_{j} - a_i}}
		\int_{q_{i j}^1(\xi)}^{q_{i j}^2(\xi)}
		\varrho(x)
		(a_i - a_j)^\top x
		\d \HH^1(x)
		,
	\end{equation*}
	where 
	$q_{i j}^1(\xi)$ and $q_{i j}^2(\xi)$
	are the vertices of the edge $\set{i,j}$.
	Since the edge $\set{i, j}$ lies between the regions
	$D_i(\xi)$ and $D_j(\xi)$,
	all points $x$ on this edge
	satisfy
	$(a_j - a_i)^\top x = \xi_j - \xi_i$.
	Thus,
	\begin{align*}
		W'(\xi;\delta \xi)
		&=
		2A^\top \Theta(\xi) \delta \xi
		-2
		\sum_{\set{i,j} \in E(\xi)}
		\frac{\delta \xi_{j} - \delta \xi_i}{\abs{a_{j} - a_i}}
		\int_{q_{i j}^1(\xi)}^{q_{i j}^2(\xi)}
		\varrho(x)
		(\xi_i - \xi_j)
		\d \HH^1(x)
		\\
		&=
		2A^\top \Theta(\xi) \delta \xi
		+2
		\sum_{\set{i,j} \in E(\xi)}
		Q_{ij}(\xi) \xi^\top (e_i - e_j) (e_i - e_j)^\top \delta \xi
		\\
		&=
		2(A-\xi)^\top \Theta(\xi) \delta \xi,
	\end{align*}
	where we reused the notation from the proof of
	\itemref{lem:properties of Theta:i}.
	This yields continuous differentiability of $W$
	and $W'(\xi) = 2(A-\xi)^\top \Theta(\xi)$.
	Since $j$ is continuously differentiable,
	the definition of $r$ implies
	$J'(\xi) = \alpha r(\xi)^\top \Theta(\xi)$.
	Further, by the chain rule $r'(\xi) = \frac 1 \alpha \nabla^2 g(\tilde u(\xi)) \Theta(\xi) - I$ is valid.
\end{proof}

We compare our differentiability result \cref{thm:J_r_tilde u_W_differentiable_for_DC}
with corresponding results from the literature.

\begin{remark}
	\label{rem:other_results}
	We mention that \cref{thm:J_r_tilde u_W_differentiable_for_DC}
	is a special case of the general \cite[Theorem~1.3]{KitagawaMerigotThibert2019}.
	Note that the derivative of the function $\Phi$ therein
	corresponds to our function $\tilde u$,
	see \cite[(1.4)]{KitagawaMerigotThibert2019},
	and \cite[(1.8)]{KitagawaMerigotThibert2019}
	is our $\Theta$ from \cref{def:theta}.
	Therein, the main assumption regarding the differentiability
	is \cite[(1.7)]{KitagawaMerigotThibert2019},
	which in our language reads $u_d(D_i(\xi)) > 0$ for all $i = 1,\ldots,n$.

	In \cite[Proposition~2]{DeGournayKahnLebrat2018},
	also the differentiability with respect to the points $a_i$
	is proved
	and the differentiability assumption is slightly relaxed to
	$D_i(\xi) \ne \emptyset$ for all $i = 1,\ldots,n$.
	Finally,
	\cite[Proposition~3.4]{LebratDeGournayKahnWeiss2019}
	proves differentiability almost everywhere
	by showing that $\tilde u$ is differentiable at all points $\xi$
	such that for all $i = 1,\ldots,n$ we have
	$D_i(\xi) \ne \emptyset$ or $\varphi_\xi(x) > a_i^\top x - \xi_i$ for all $x \in D$.
	One can check that this condition is implied by \eqref{DC}.

	We included the above proof of \cref{thm:J_r_tilde u_W_differentiable_for_DC},
	since it is simpler than the proofs of
	\cite{KitagawaMerigotThibert2019,DeGournayKahnLebrat2018}.
\end{remark}

\begin{remark}
	\label{rem:differentiability}
	By means of some simple examples,
	we want to point out the differences between \eqref{DC}
	and the differentiability assumptions from
	\cite{KitagawaMerigotThibert2019,DeGournayKahnLebrat2018,LebratDeGournayKahnWeiss2019}.
	We choose $D = [-1,1]^2$.

	In the first example,
	we choose $n = 4$,
	\begin{equation*}
		a_1 =
		\begin{pmatrix}
			1 \\ 1
		\end{pmatrix}
		,\;
		a_2 =
		\begin{pmatrix}
			-1 \\ 1
		\end{pmatrix}
		,\;
		a_3 =
		\begin{pmatrix}
			0 \\ -1
		\end{pmatrix}
		,\;
		a_4 =
		\begin{pmatrix}
			0 \\ 0
		\end{pmatrix}
	\end{equation*}
	and $\xi \equiv 0$.
	For $\bar x = (0,0) \in D$ we have
	$I_\xi(\bar x) = \set{1,2,3,4}$,
	which shows that
	\eqref{DC}
	does not hold.
	Moreover,
	it can be checked that $D_4(\xi) = \emptyset$,
	but
	$\varphi_\xi(\bar x) = 0 = a_4^\top \bar x - \xi_4$,
	therefore the differentiability assumption from 
	\cite[Proposition~3.4]{LebratDeGournayKahnWeiss2019}
	is also violated.
	By a detailed analysis one can check that $\tilde u$
	is still differentiable at $\xi$, since
	$\abs{\tilde u_4(\xi + h) - \tilde u_4(\xi)} = \tilde u_4(\xi + h) = \mathcal{O}(h^2)$
	as $h \to 0$.

	Next, we consider $n = 4$,
	\begin{equation*}
		a_1 =
		\begin{pmatrix}
			1 \\ 0
		\end{pmatrix}
		,\;
		a_2 =
		\begin{pmatrix}
			-1 \\ 0
		\end{pmatrix}
		,\;
		a_3 =
		\begin{pmatrix}
			0 \\ 1
		\end{pmatrix}
		,\;
		a_4 =
		\begin{pmatrix}
			0 \\ -1
		\end{pmatrix}
	\end{equation*}
	and $\xi \equiv 0$.
	Again, 
	$I_\xi(\bar x) = \set{1,2,3,4}$
	and \eqref{DC} is violated.
	However, $D_i(\xi) \ne \emptyset$ for all $i$.
	Therefore,
	the conditions from
	\cite{KitagawaMerigotThibert2019,LebratDeGournayKahnWeiss2019}
	hold
	and this yields differentiability.

	In the third example,
	we choose $n = 3$,
	\begin{equation*}
		a_1 =
		\begin{pmatrix}
			1 \\ 0
		\end{pmatrix}
		,\;
		a_2 =
		\begin{pmatrix}
			-1 \\ 0
		\end{pmatrix}
		,\;
		a_3 =
		\begin{pmatrix}
			0 \\ 0
		\end{pmatrix}
	\end{equation*}
	and $\xi \equiv 0$.
	Now,
	$D_3(\xi) = \emptyset$,
	but we still have
	$\varphi_\xi(x) = 0 = a_3^\top x - \xi_3$
	for all $x$ belonging to
	the line segment $\set{0} \times [-1, 1]$.
	Further, we can check that
	$\tilde u_3(0,0,s) = -4 s$ for $s \in [-1,0]$
	and
	$\tilde u_3(0,0,s) = 0$ for $s \ge 0$.
	This shows that $\tilde u$ is not differentiable at $\xi$
	and, consequently, all the differentiability conditions are violated.

	Finally,
	we just mention that
	the situation from this last example is the only source of
	nondifferentiability.
	In fact,
	nondifferentiability occurs
	if and only if
	two regions $D_i(\xi)$ and $D_j(\xi)$,
	with $i,j \in \set{-k,\ldots,-1} \cup \set{1,\ldots, n}$,
	share an edge
	and a third index $k \not\in \set{i, j}$, $k > 0$,
	exists with $\varphi_\xi(x) = a_k^\top x - \xi_k$
	on the edge between $D_i(\xi)$ and $D_j(\xi)$.
	This characterization will be of no importance for the remainder.
\end{remark}

\begin{theorem}
	\label{thm:J_r_tilde_u_W_PC1_semismooth}
	The functions $J, r, \tilde u, W$ are $PC^1$ and semismooth.

	Further, for all $\xi \in \R^n$,
	$\Theta(\xi) \in \partial^C \tilde u(\xi)$,
	$\frac 1 \alpha \nabla^2 g(\tilde u(\xi)) \Theta(\xi) - I \in \partial^C r(\xi)$,
	$2(A-\xi)^\top \Theta(\xi) \in \partial^C W(\xi)$ and
	$\alpha r(\xi)^\top \Theta(\xi) \in \partial^C J(\xi)$.
\end{theorem}
\begin{proof}
	We consider the function $W$.
	The argument for the other functions is similar.

	We already know that $W$ is $C^1$ on each sector $\Xi^k$,
	see \cref{thm:J_r_tilde u_W_differentiable_for_DC}.
	Our goal is to extend the function $W$ from $\Xi^k$
	to a $C^1$ function $\Phi_k$ on $\R^n$.
	This implies $W(\xi) \in \set{\Phi_k(\xi) \given k = 1,\ldots,N}$
	for all $\xi \in \bigcup_{k = 1}^N \cl(\Xi^k) = \R^n$,
	i.e., $W$ is $PC^1$.
	For the extension of the function, we use
	Whitney's extension theorem,
	\cite[Thm.~I]{WhitneyExtension}.

	Let $\Xi^k$ be a sector.
	The function $W$ is Lipschitz continuous on $\cl(\Xi^k)$
	by \cref{thm:J_r_W_tilde_u_Lipschitz}.
	The derivative on $\Xi^k$ is $W'(\xi) = 2(A-\xi)^\top \Theta(\xi)$, see \cref{thm:J_r_tilde u_W_differentiable_for_DC}.
	Since $\Theta$ is uniformly continuous on $\Xi^k$,
	see \itemref{lem:properties of Theta:iii},
	it can be extended continuously to $\cl(\Xi^k)$
	and, consequently,
	the derivative $W'$ can be extended to a continuous function
	$\Psi_k$
	on $\cl(\Xi^k)$.

	In order to apply
	\cite[Thm.~I]{WhitneyExtension},
	we have to check the Taylor-like prerequisites.
	For any $\xi, \tilde \xi \in \Xi^k$,
	the differentiability of $W$ on the open and convex set $\Xi^k$
	implies
	the existence of $t \in [0,1]$ with
	\begin{equation*}
		W(\tilde \xi)
		=
		W(\xi)
		+
		W'(\xi + t (\tilde \xi - \xi)) (\tilde \xi - \xi).
	\end{equation*}
	By continuity and since $[0,1]$ is compact,
	for any $\xi, \tilde \xi \in \cl(\Xi^k)$,
	there exists $t \in [0,1]$ with
	\begin{equation*}
		W(\tilde \xi)
		=
		W(\xi)
		+
		\Psi_k(\xi + t (\tilde \xi - \xi)) (\tilde \xi - \xi).
	\end{equation*}
	This implies
	\begin{equation*}
		W(\tilde \xi)
		=
		W(\xi)
		+
		\Psi_k(\xi) (\tilde \xi - \xi)
		+
		R(\tilde \xi, \xi)
	\end{equation*}
	with
	\begin{equation*}
		R(\tilde \xi, \xi)
		=
		(\Psi_k(\xi + t (\tilde \xi - \xi)) - \Psi_k(\xi)) (\tilde \xi - \xi)
		.
	\end{equation*}
	Owing to the continuity of $\Psi_k$ on $\cl(\Xi^k)$,
	for each $\xi_0 \in \cl(\Xi^k)$,
	we have
	\begin{equation*}
		R(\tilde \xi, \xi)
		=
		\oo(\abs{\tilde \xi - \xi})
		\qquad
		\text{as } \xi, \tilde \xi \to \xi_0
		.
	\end{equation*}
	Thus, we can apply
	\cite[Thm.~I]{WhitneyExtension}
	and obtain a $C^1$ function $\Phi_k$ on $\R^n$
	which extends $W$.
	As explained above, this shows that $W$
	is $PC^1$.

	It remains to check that the given expression belong to
	the Clarke subdifferential.
	Let $\xi \in \R^n \setminus \bigcup_{i = 1}^N \Xi^i$ be arbitrary.
	We only verify $\Theta(\xi) \in \partial^C \tilde u(\xi)$.
	The other functions can be treated analogously.
	By \cref{thm:degenerated sectors are zero-set},
	the index set $\hat I := \set{i \in \set{1,\ldots,N} \given \xi \in \cl(\Xi^i)}$ is nonempty.
	We choose $i \in \hat I$, such that
	the cardinality of the configuration $k(\Xi^i)$ (see \cref{def:neu}) is minimal.
	Now, we choose an arbitrary $\hat \xi \in \Xi^i$ and consider
	$\xi_t := t \hat \xi + (1 - t) \xi$ for $t \in (0,1)$.
	Since $\Xi^i$ is open and convex, $\xi_t \in \Xi^i$ for all $t \in (0,1)$
	and we clearly have $\xi_t \to \xi$ as $t \searrow 0$.
	It remains to show that $\Theta(\xi_t) \to \Theta(\xi)$.
	Note that this does not simply follow from
	\itemref{lem:properties of Theta:iii}, since $\xi \not\in \Xi^i$.
	Due to the minimality of the cardinality of $k(\Xi^i)$,
	a region $D_i(\xi_t)$ is nonempty if and only if $D_i(\xi)$
	is nonempty, in particular, no new regions could appear.
	Since the coordinates $q_{i j l}(\xi_t)$
	appearing in the definition of $\Theta(\xi_t)$
	depend linearly on $t$, see \cref{def:q_ijl},
	we can check $\lim_{t \searrow 0} \Theta(\xi_t) = \Theta(\xi)$.
	Note that it can happen that a region $D_i(\xi_t)$ has more edges than the corresponding $D_i(\xi)$,
	but the lengths of these additional edges converge to $0$
	and, therefore, the additional integral appearing in the definition of $\Theta_{ik}(\xi_t)$
	converges to $0$ as well.
	The claim follows from the definition of the Clarke subdifferential in \cref{def:Clarke-Ableitung} and the differentiability in $\Xi^i$, see \cref{thm:J_r_tilde u_W_differentiable_for_DC}.
\end{proof}

\begin{remark}
	\label{rem:PC1}
	We emphasize that the partitioning of $\R^n$ into the sectors $\Xi^i$
	plays a crucial role in the proof of \cref{thm:J_r_tilde_u_W_PC1_semismooth}.
	We note that
	it does not simply follow from the almost-everywhere differentiability
	provided \cite[Proposition~3.4]{LebratDeGournayKahnWeiss2019}
	that the functions in \cref{thm:J_r_tilde_u_W_PC1_semismooth}
	are indeed $PC^1$.

	We illustrate this with an example.
	We consider $D = [-1, 3] \times [2, 2]$, $n = 5$,
	\begin{equation*}
		a_1 =
		\begin{pmatrix}
			0 \\ 0
		\end{pmatrix}
		,\;
		a_2 =
		\begin{pmatrix}
			2 \\ 1
		\end{pmatrix}
		,\;
		a_3 =
		\begin{pmatrix}
			1 \\ 2
		\end{pmatrix}
		,\;
		a_4 =
		\begin{pmatrix}
			-1 \\ -2
		\end{pmatrix}
		,\;
		a_5 =
		\begin{pmatrix}
			-3 \\ -1
		\end{pmatrix}
		.
	\end{equation*}
	For the presentation, we restrict $\xi$ to the two-dimensional subspace $\R^2 \times \set{0}^3$.
	The regions $D_i(\bar \xi)$ for $\bar \xi = (0, 1, 0, 0, 0)$
	are shown on the left-hand side of \cref{fig:differentiability}.
	Note that we have $D_1(\bar \xi) = \emptyset$,
	but $\varphi_{\bar \xi}(x) = a_1^\top x - \bar \xi_1$ for all $x \in \cl(D_3(\bar\xi)) \cap \cl(D_4(\bar\xi))$.
	Consequently, $W$ is not differentiable at $\bar \xi$.
	In fact, the points of non-differentiability within $\R^2 \times \set{0}^3$
	are precisely $\set{0} \times (0,\infty) \times \set{0}^3$.
	Note that for $\xi_2 \le 0$, the problematic middle edge
	$\cl(D_3(\xi)) \cap \cl(D_4(\xi))$
	vanishes.
	On the right-hand side of \cref{fig:differentiability},
	we show the function
	$(\xi_1, \xi_2) \mapsto W(\xi_1, \xi_2, 0, 0, 0) - 2.5 \xi_2^2$.
	The quadratic modification of the function $W$ enhances the visibility of the nondifferentiability.
	We can see that the points of differentiability within the region
	$[-1/4,1/4] \times [-1, 2] \times \set{0}^3$
	is a connected set.
	In particular, it is not possible to extend $W'$
	in a continuous way to the points at which $W$ is not differentiable.
	In the proof of \cref{thm:J_r_tilde_u_W_PC1_semismooth},
	we used the definition of the sectors $\Xi^i$
	in order to partition the points of differentiability into smaller, convex subsets
	and this allows us to show that $W$ is indeed $PC^1$.
\end{remark}
\begin{figure}[ht]
	\centering
	\includegraphics[scale=0.14]{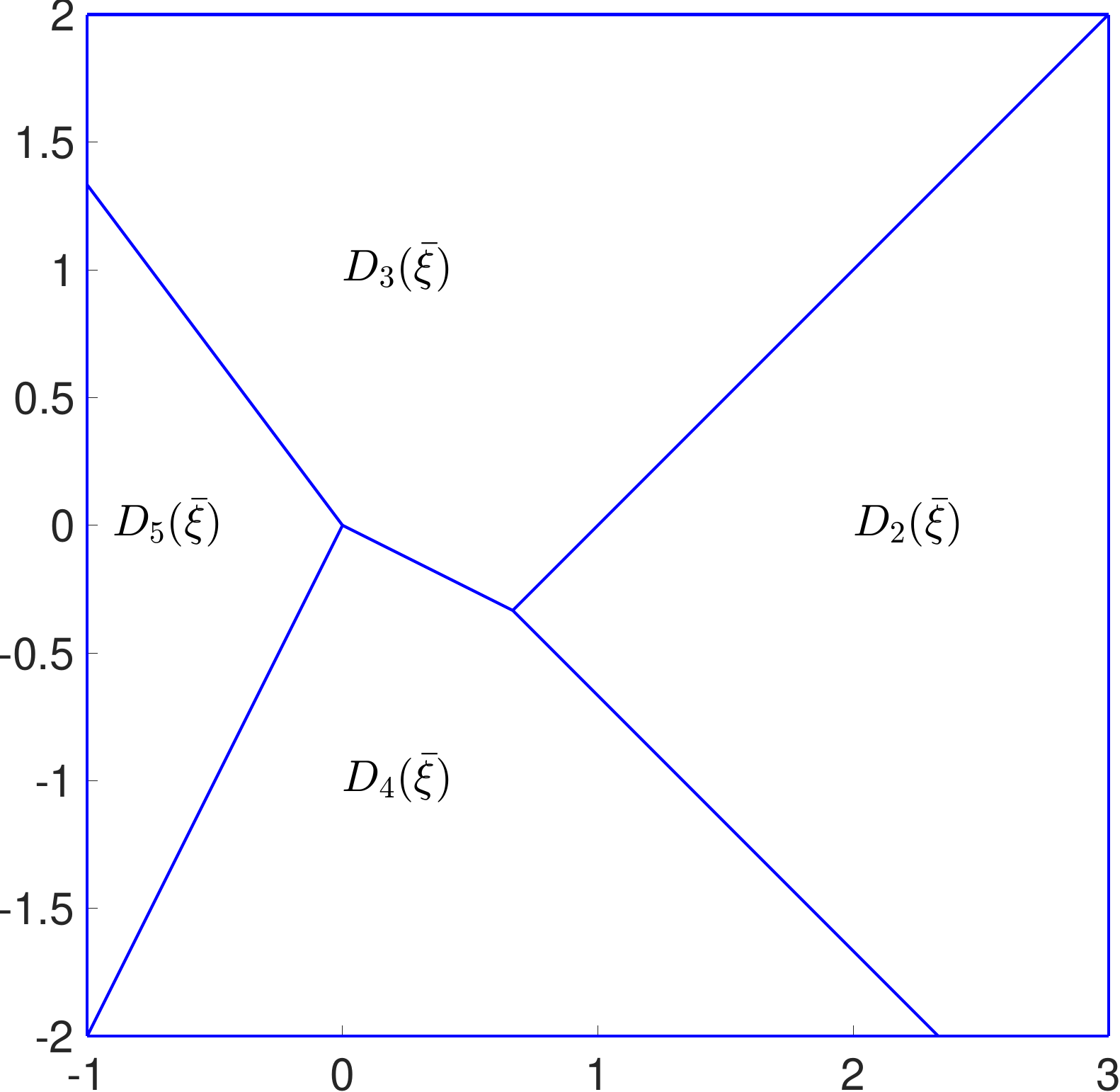}%
	\hspace{1cm}%
	\includegraphics[scale=0.14]{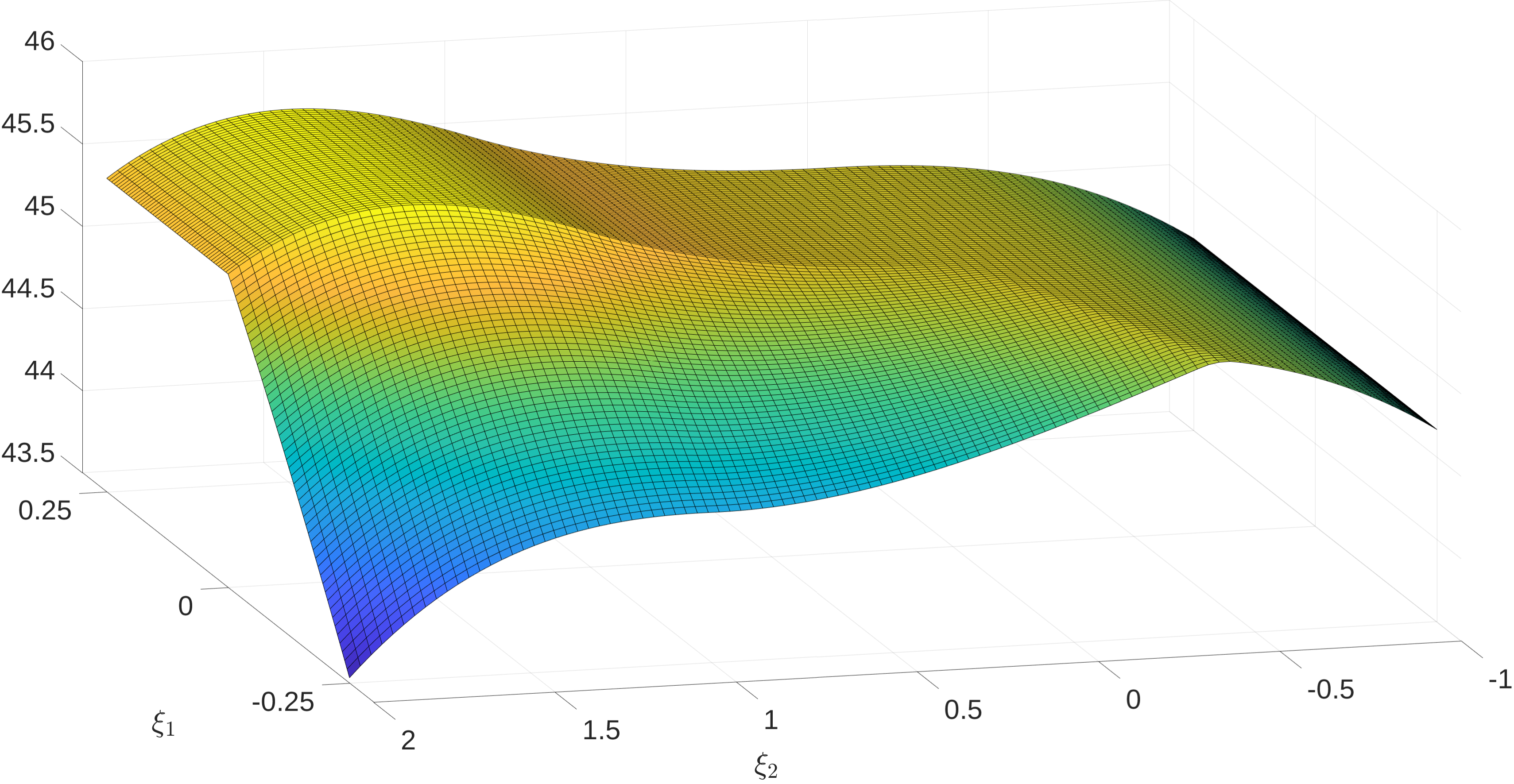}
	\caption{%
		In the left plot, we visualize the regions $D_i(\bar\xi)$ for $\bar\xi = (0,2,0,0,0)$.
		The nondifferentiability of $W(\cdot, \cdot, 0,0,0)$ is shown in the right plot.
	}
	\label{fig:differentiability}
\end{figure}

\subsection{The case \texorpdfstring{$u_d = \lambda^2$}{the Lebesgue measure}}
\label{subsec:Lebesgue_case}

In this section, we consider the important case that
$u_d$ is (the restriction of) the Lebesgue measure $\lambda^2$
on $D$, i.e.,
the density function satisfies $\varrho \equiv 1$
on $D$.
This allows us to simplify some of the integrals.

For the derivative $\Theta$ of $\tilde u$,
we can use the representation from
the proof of \itemref{lem:properties of Theta:i}
in terms of the inner edges.
This gives
\begin{equation*}
	\Theta(\xi)
	=
	-\sum_{\set{i,j} \in E(\xi)} \frac{\ell_{i j}}{\abs{a_i - a_j}}
	(e_i - e_j) (e_i - e_j)^\top,
\end{equation*}
where $\ell_{i j}$ is the length of the edge.

We consider the integrals in $W$, see \eqref{eq:W(xi) as sum of integrals over regions}.
First, the integral $\int_D \abs{x}^2 \d x$
can be calculated easily since $D$ is a polygon.
Second, we consider an integral with a constant integrand on an arbitrary (convex) polygon $P$
with boundary vertices
$q_1, \ldots, q_m = q_0$ (in counterclockwise order).
The function $f\colon \R^2 \to \R^2, x \mapsto \frac 1 2 x$,
satisfies $\divergence f = 1$.
Gauß' theorem and exact evaluation of the boundary integrals by the midpoint rule yield
\begin{align*}
	\int_P 1 \d x
	&=
	\int_{\partial P} f(x)^\top n(x) \d \HH^1(x)
	=
	\frac 1 2 \sum_{i=1}^m
	\int_{q_{i-1}}^{q_i}
	\frac{x^\top}{\abs{q_i - q_{i-1}}}
	\left(
		\begin{matrix}
			q_{i, 2} - q_{i-1, 2} \\
			q_{i-1, 1} - q_{i, 1}
		\end{matrix}
	\right)
	\d \HH^1(x)
	\\
	&=
	\frac 1 4 \sum_{i=1}^m
	\left(
		\begin{matrix}
			q_{i-1, 1} + q_{i, 1} \\
			q_{i-1, 2} + q_{i, 2}
		\end{matrix}
	\right)
	^\top
	\left(
		\begin{matrix}
			q_{i, 2} - q_{i-1, 2} \\
			q_{i-1, 1} - q_{i, 1}
		\end{matrix}
	\right)
	=
	\frac 1 2 \sum_{i=1}^m
	q_{i-1, 1} q_{i, 2} - q_{i, 1} q_{i-1, 2}
	,
\end{align*}
which is the so-called shoelace formula.
Third, we compute integrals of type $\int_{P} x^\top v \d x$ with $v \in \R^2$.
We set $f: \R^2 \to \R, x \mapsto x_1 x_2 (v_2, v_1)^\top$ and therefore
$\divergence f = x^\top v$.
Gauß' theorem gives
\[
	\int_{P} x^\top v \d x
	=
	\int_{\partial P} f(x)^\top n(x) \d \HH^1(x)
	=
	\sum_{i=1}^m
	\int_{q_{i-1}}^{q_i}
	f(x)^\top
	n(x)
	\d \HH^1(x)
	.
\]
As $f$ is quadratic on the edges, Simpson's rule yields the exact integral
\begin{align*}
	\MoveEqLeft[1]
	\int_{q_{i-1}}^{q_i} f(x)^\top n(x) \d \HH^1(x)
	=
	\left(
		\begin{matrix}
			v_2 \\
			v_1
		\end{matrix}
	\right)^\top
	\frac{1}{\abs{q_i - q_{i-1}}}
	\left(
		\begin{matrix}
			q_{i, 2} - q_{i-1, 2} \\
			q_{i-1, 1} - q_{i, 1}
		\end{matrix}
	\right)
	\int_{q_{i-1}}^{q_i} x_1 x_2 \d \HH^1(x)
	\\
	&=
	\frac 1 6
	\left(
		\begin{matrix}
			v_2 \\
			v_1
		\end{matrix}
	\right)^\top
	\left(
		\begin{matrix}
			q_{i, 2} - q_{i-1, 2} \\
			q_{i-1, 1} - q_{i, 1}
		\end{matrix}
	\right)
	\left(
		q_{i-1, 1} q_{i-1, 2} +
		4 \frac{q_{i-1, 1} + q_{i, 1}}{2} \frac{q_{i-1, 2} + q_{i, 2}}{2} +
		q_{i, 1} q_{i, 2}
	\right)
	\\
	&=
	\frac 1 6
	\left(
		\begin{matrix}
			v_2 \\
			v_1
		\end{matrix}
	\right)^\top
	\left(
		\begin{matrix}
			q_{i, 2} - q_{i-1, 2} \\
			q_{i-1, 1} - q_{i, 1}
		\end{matrix}
	\right)
	\left(
		2 q_{i,1} q_{i,2}
		+ 2 q_{i-1,1} q_{i-1,2}
		+ q_{i-1,1} q_{i,2}
		+ q_{i,1} q_{i-1,2}
	\right)
	.
\end{align*}
By summing up these integrals, some terms cancel out
and we get
\begin{align*}
	\int_{P} x^\top v \d x
	&=
	\sum_{i=1}^m
	\frac{v_2 (q_{i, 2} + q_{i-1, 2}) + v_1 (q_{i, 1} + q_{i-1, 1})}{6}
	\left(
		q_{i-1,1} q_{i,2} - q_{i,1} q_{i-1,2}
	\right)
	\\
	&=
	\sum_{i=1}^m
	\frac{v^\top (q_{i-1} + q_{i})}{3}
	\frac{q_{i-1,1} q_{i,2} - q_{i,1} q_{i-1,2}}{2}
	.
\end{align*}
For the calculation of $W$,
we sum up these integrals
with $P = D_i(\xi)$ and $v = a_i$
for all $i \in \set{1,\ldots,n}$.
For practical reasons we write this in terms of the edges, where $E$ denotes all edges.
Let $e \in E$ be an edge with vertices $q_1^e, q_2^e$.
Let $a_l, a_r \in \R^2$ be the vectors $a_i$ on the left/right side of $e$.
We set $a_l = 0$ or $a_r = 0$ for the non-existing side of boundary edges.
Previously,
every edge was considered twice,
thus we can write
\begin{align*}
	\begin{split}
		\sum_{i=1}^n \int_{D_i} x^\top a_i \d x
		=
		\frac 1 3
		\sum_{e \in E}
		\left(
			(a_l - a_r)^\top (q_1^e + q_2^e)
		\right)
		\frac{q_{1,1}^e q_{2,2}^e - q_{2,1}^e q_{1,2}^e}{2}
		.
	\end{split}
\end{align*}
For an inner edge,
the addend can be rewritten as
\(
	(\xi_l - \xi_r)
	\parens{q_{1,1}^e q_{2,2}^e - q_{2,1}^e q_{1,2}^e}
	.
\)

\section{Numerical algorithms}
\label{sec:algorithms}
In this section, we address algorithms for the solution of \eqref{P}.
We briefly argue that it seems not be a good idea to solve \eqref{P} directly,
although the function $W_2^2(u_d, \cdot)$ appearing in \eqref{P}
is convex
(and even differentiable under slight assumptions, see \cite[Propositions~7.17 and 7.18]{Santambrogio}).
The reason is that the evaluation of $W_2^2(u_d, u)$ for a given $u$ is very time consuming,
see the discussion after \eqref{eq:eval_transport_dist}.
Consequently, our algorithms work on problem \eqref{P(xi)}.

\subsection{Fixed-point iteration}
\label{subsec:fixed_point}

As an optimality condition for \eqref{P(xi)} we derived $r(\xi) = 0$, cf.\ \cref{lem:optimality of xi on J}.
This can be transformed into the fixed-point equation
$\xi = \xi + \tau r(\xi)$ with $\tau > 0$.
A possible algorithm for the solution
is the fixed-point iteration
$\xi_{k+1} := \xi_k + \tau_k r(\xi_k)$
with appropriate step sizes $\tau_k > 0$.
It is not clear, how to choose these step sizes.
We will see in \cref{lem:r(xi) is descent AND J(xi(t)) monotonically decreasing} below,
that the direction $r(\xi_k)$ is a non-ascent direction for $J$.
However, since $J$ is not $C^1$, Armijo step sizes could become
arbitrarily small and the usual convergence proof does not work.
Similarly, constant (but small) step sizes might not give a
decrease of $J$.

Instead, we can view the above iteration scheme
as a discretization of an ODE.
Indeed, the above update formula yields
$(\xi_{k+1} - \xi_k)/\tau_k = r(\xi_k)$.
This is an explicit Euler scheme for the ODE
\begin{align}
	\label{ODE(xi)}
	\tag{ODE\mbox{$_\xi$}}
	\begin{split}
		\xi'(t) &= r(\xi(t)),
		\\
		\xi(0) &= \xi_0.
	\end{split}
\end{align}

We state some simple properties for
\eqref{ODE(xi)}.

\begin{lemma}
	\label{lem:r Lipschitz AND xi-ODE unique solution AND difference of ODE solutions}
	The following properties hold.
	\begin{enumerate}
		\item The function $r$ is globally Lipschitz continuous with constant $L > 0$.
			\label{lem:r Lipschitz AND xi-ODE unique solution AND difference of ODE solutions:i}
		\item There exists a unique solution of \eqref{ODE(xi)} on $[0,\infty)$.
			\label{lem:r Lipschitz AND xi-ODE unique solution AND difference of ODE solutions:ii}
		\item Let $\xi, \tilde \xi$ be the solutions of \eqref{ODE(xi)} for given initial data $\xi_0, \tilde \xi_0$.
			Then, $\abs{\xi(t) - \tilde \xi(t)} \le \abs{\xi(0) - \tilde \xi(0)} e^{Lt}$ holds for $t \ge 0$.
			\label{lem:r Lipschitz AND xi-ODE unique solution AND difference of ODE solutions:iii}
	\end{enumerate}
\end{lemma}
\begin{proof}
	\ref{lem:r Lipschitz AND xi-ODE unique solution AND difference of ODE solutions:i}:
	This has been proven in \Cref{thm:J_r_W_tilde_u_Lipschitz}.

	\ref{lem:r Lipschitz AND xi-ODE unique solution AND difference of ODE solutions:ii}:
	This follows from the Picard--Lindelöf theorem.

	\ref{lem:r Lipschitz AND xi-ODE unique solution AND difference of ODE solutions:iii}:
	This can be done as usual using Gronwall's lemma.
\end{proof}
As announced, $r(\xi)$ yields a non-ascent direction for $J$.
\begin{lemma}
	\label{lem:r(xi) is descent AND J(xi(t)) monotonically decreasing}
	For every $\hat \xi \in \R^n$,
	we have $J'(\hat \xi;r(\hat \xi)) \le 0$.
	Further, let $\xi \colon [0,\infty) \to \R^n$ solve \eqref{ODE(xi)}.
	Then, $J(\xi(t))$ is monotonically decreasing in $t$
	and $J(\xi(t)) \to J_0$ as $t \to \infty$ for some $J_0 \in \R$.
\end{lemma}
\begin{proof}
	Since $\cl(\Xi^k)$ are closed and convex sets which cover $\R^n$,
	there exist $k \in \set{1,\ldots,N}$ and $\hat h > 0$
	such that
	$\hat \xi + h r(\hat \xi) \in \cl(\Xi^k)$ for all $h \in [0,\hat h]$.
	The proof of \cref{thm:J_r_tilde_u_W_PC1_semismooth} yields
	a function $J^k \in C^1(\R^n)$ with $J = J^k$
	on the set $\cl(\Xi^k)$.
	Consequently,
	\begin{equation*}
		J'(\hat \xi;r(\hat \xi)) 
		=
		\lim_{h \searrow 0} \frac{J^k(\hat \xi + h r(\hat \xi)) - J^k(\hat \xi)}{h}
		=
		(J^k)'(\hat \xi)r(\hat \xi)
		.
	\end{equation*}
	For a sequence
	$\seq{\hat \xi_i}_i \subset \Xi^k$ with $\hat \xi_i \to \hat \xi$
	we get
	\begin{align*}
		J'(\hat \xi;r(\hat \xi))
		= 
		(J^k)'(\hat \xi)r(\hat \xi)
		=
		\lim_{i \to \infty} (J^k)'(\hat \xi_i) r(\hat \xi_i)
		=
		\lim_{i \to \infty} \alpha r(\hat \xi_i)^\top \Theta(\hat \xi_i) r(\hat \xi_i)
		\le
		0
		,
	\end{align*}
	where we used \cref{thm:J_r_tilde u_W_differentiable_for_DC},
	\cref{thm:J_r_W_tilde_u_Lipschitz}
	and
	\itemref{lem:properties of Theta:i}.

	Let $\xi$ be the solution of \eqref{ODE(xi)}.
	Since $J$ is Lipschitz continuous by \cref{thm:J_r_W_tilde_u_Lipschitz},
	we can use the chain rule for directionally differentiable functions
	and find that the
	right directional derivative of
	$t \mapsto J(\xi(t))$ coincides with $J'(\xi(t); r(\xi(t)))$.
	Consequently,
	the first part of the proof shows that this
	directional derivative is nonpositive.
	Hence, the function $t \mapsto J(\xi(t))$ is decreasing.
	Since $J$ is bounded from below, see \cref{thm:uniqueness},
	this also implies the convergence.
\end{proof}

At this point, we would like to mention that \eqref{ODE(xi)}
is not a (sub)-gradient flow.
Indeed, at points $\xi$ satisfying \eqref{DC},
we have 
$r'(\xi) = \frac1\alpha \nabla^2 g(\tilde u(\xi)) \Theta(\xi) - I$
and this matrix is, in general, not symmetric;
consequently, $r$ cannot be a gradient.

Further, the map $J$ is in general not convex,
since the map $W$ is not convex, see also \cref{fig:differentiability}.
Thus, it is also not amenable to established convex optimization methods.

\begin{theorem}
	The solution of	\eqref{ODE(xi)} is bounded.
	In particular, $\xi$ has an accumulation point,
	i.e., there exists a sequence $\seq{t_i} \subset (0,\infty)$
	with $t_i \to \infty$
	and $\xi(t_i) \to \tilde \xi_0$ for some $\tilde \xi_0 \in \R^n$.
\end{theorem}
\begin{proof}
	We rewrite the ODE as $\xi + \xi' = A + \frac 1 \alpha \nabla g(\tilde u(\xi))$.
	The right-hand side is bounded by a constant $\kappa$ using Weierstraß' theorem
	as $g$ is $C^2$, $\tilde u$ is continuous by \cref{thm:J_r_W_tilde_u_Lipschitz} and $\Uad$ is compact.
	Using the fundamental theorem of calculus on $t \mapsto e^t \xi(t)$
	we get
	\begin{align*}
		\abs{\xi(t)}
		&\le
		e^{-t} \left( e^0 \abs{\xi_0} + \int_0^t \abs*{e^s (\xi(s) + \xi'(s))} \d s \right)
		\le
		\abs{\xi_0} + e^{-t} \int_0^t e^s \kappa \d s
		\le
		\abs{\xi_0} + \kappa
		.
	\end{align*}
\end{proof}

Now, we are going to check
that accumulation points correspond to solutions.

\begin{theorem}
	\label{thm:conv_ODE}
	Let $\xi: [0,\infty) \to \R^n$ solve \eqref{ODE(xi)}.
	Then,
	$\operatorname{dist}( \xi(t), \tilde W) \to 0$
	as $t \to \infty$, where
	$\tilde W := \set{ \tilde \xi_0 \in \R^n \given \tilde u(\tilde \xi_0) \text{ satisfies } \eqref{eq:optimality} }$.
	If all roots of $r$ are isolated, then $\xi$ converges to a root of $r$.
\end{theorem}
Recall that the convexity of $g$ implies the uniqueness of the root of $r$,
i.e., this root is isolated, see \cref{lem:optimality of xi on J}.

\begin{proof}
	In order to prove $\operatorname{dist}( \xi(t), \tilde W) \to 0$,
	it is sufficient to check that every accumulation point
	$\tilde \xi_0$ belongs to $\tilde W$.
	For an arbitrary accumulation point $\tilde \xi_0$,
	let $\tilde \xi \colon [0,\infty) \to \R^n$ be the solution of the ODE
	\begin{align}
		\label{eq:ODE tilde xi}
		\begin{split}
			\tilde \xi'(t) &= r(\tilde \xi(t)),
			\\
			\tilde \xi(0) &= \tilde \xi_0.
		\end{split}
	\end{align}
	Since $\tilde \xi_0$ is an accumulation point,
	we have $\xi(t_k) \to \tilde \xi_0$ as $k \to \infty$
	for some $t_k \to \infty$.
	Thus,
	we can apply
	\itemref{lem:r Lipschitz AND xi-ODE unique solution AND difference of ODE solutions:iii}
	to
	$\xi(t_k + \cdot)$ and $\tilde \xi(\cdot)$
	and get
	\begin{equation*}
		\xi(t_k + s) \to \tilde \xi(s)
		\qquad\text{for } k \to \infty
	\end{equation*}
	for all $s > 0$.
	Since $J$ is continuous, this yields
	\begin{equation*}
		J(\xi(t_k + s)) \to J(\tilde \xi(s))
		\qquad\text{for } k \to \infty
	\end{equation*}
	for all $s \ge 0$.
	However, from \cref{lem:r(xi) is descent AND J(xi(t)) monotonically decreasing},
	we get
	$J(\xi(t_k + s)) \to J_0$
	and
	$J(\xi(t_k)) \to J_0$.
	Putting everything together, this yields
	$J(\tilde \xi_0) = J(\tilde \xi(s))$
	for all $s > 0$.
	In particular,
	$J'(\tilde \xi(t);r(\tilde \xi(t))) = 0$ holds.

	Next, we check that $\tilde u(\tilde \xi(t))$
	is constant as well.
	We fix some $t \ge 0$.
	There exists a sequence $\seq{h_i}_i \subset (0,\infty)$
	with $h_i \to 0$
	such that
	$\tilde \xi(t) + h_i r(\tilde \xi(t)) \in \cl(\Xi^k)$
	for some fixed $k \in \set{1,\ldots,N}$.
	Consequently,
	\begin{equation*}
		0
		=
		J'(\tilde \xi(t); r(\tilde \xi(t)))
		=
		\lim_{i \to \infty}
		\frac{J^k(\tilde \xi(t) + h_i r(\tilde \xi(t))) - J^k(\tilde \xi(t))}{h_i}
		=
		(J^k)'( \tilde \xi(t) ) r(\tilde \xi(t))
		.
	\end{equation*}
	Here, $J^k$ are $C^1$ extensions from
	the restriction of $J$ to $\Xi^k$,
	see the proof of \cref{thm:J_r_tilde_u_W_PC1_semismooth}.
	From
	\cref{lem:properties of Theta,thm:J_r_tilde u_W_differentiable_for_DC}
	we get
	\begin{equation*}
		0
		=
		J'(\tilde \xi(t); r(\tilde \xi(t)))
		=
		\alpha r(\tilde \xi(t))^\top \Theta_k(\tilde \xi(t)) r(\tilde \xi(t))
		,
	\end{equation*}
	where $\Theta_k$ is the continuous extension to $\cl(\Xi^k)$
	of the restriction of $\Theta$ to $\Xi^k$.
	Similarly,
	we get
	\begin{equation*}
		\tilde u'(\tilde \xi(t); r(\tilde \xi(t)))
		=
		\Theta_k(\tilde \xi(t)) r(\tilde \xi(t)).
	\end{equation*}
	Since the matrix
	$\Theta_k(\tilde \xi(t))$
	is symmetric and negative semidefinite,
	combining the previous two equations
	yields
	$\tilde u'(\tilde \xi(t); r(\tilde \xi(t))) = 0$,
	i.e.,
	$\tilde u(\tilde \xi(\cdot)) =: \bar u$ is constant.
	Consequently,
	\eqref{eq:ODE tilde xi}
	implies
	\begin{equation*}
		\tilde \xi'(t) + \tilde \xi(t)
		=
		A + \frac1\alpha \nabla g(\bar u)
		=:
		\bar \xi
	\end{equation*}
	and the solution of this ODE satisfies
	\begin{equation*}
		\tilde \xi(t)
		\to
		\bar \xi
		.
	\end{equation*}
	This gives
	$\bar u = \tilde u(\bar \xi)$
	and
	$r(\bar \xi) = 0$.
	\Cref{lem:optimality of xi on J}
	implies that $\bar u = \tilde u(\tilde \xi_0)$ satisfies \eqref{eq:optimality}.

	Now, let solutions of $r(\bar \xi) = 0$ be isolated.
	We set $\overline{W} := \set{\bar \xi \in \R^n \given r(\bar \xi) = 0}$
	and
	$a(\xi) := A + \frac1\alpha \nabla g(\tilde u(\xi))$.
	The first part of the proof shows that
	$r(a(\tilde \xi_0)) = 0$
	holds for every
	accumulation point $\tilde \xi_0$.
	Consequently,
	every accumulation point of $a(\xi(t))$ lies in $\overline{W}$.
	As $\xi$ and $a$ are continuous and $\overline{W}$ contains only isolated points,
	it exists $\bar \xi \in \overline{W}$ with $a(\xi(t)) \to \bar \xi$.
	Therefore,
	$\xi'(t) + \xi(t) = a(\xi(t)) \to \bar \xi$, i.e., there is a function $R: [0,\infty) \to \R$ with $R(t) \to 0$ and $\xi'(t) + \xi(t) = \bar \xi + R(t)$.
	Consequently,
	\begin{equation*}
		\xi(t)
		=
		\bar \xi
		+
		e^{-t} \parens*{
			\xi(0)
			+
			\int_0^t R(s) e^s \d s
		}
		\to
		\bar \xi
		\in
		\overline{W}
	\end{equation*}
	and this shows the claim.
\end{proof}
This result is similar to LaSalle's invariance principle \cite[Thm.~3]{LaSalle}.
In our case, $J$ is similar to a Lyapunov function as $\frac \d \dt J(\xi(t)) \le 0$ but $J$ is not differentiable.
\cite[Thm.~3]{LaSalle} states that bounded solutions approach the set $M$
which is the set of starting points $\tilde \xi_0$ whose respective solutions $\tilde \xi(t)$ fulfill $J(\tilde \xi(t)) = C$ for some constant $C \in \R$ and whose solutions do not leave $M$.
In our case, this is just $\tilde W$.

Although the above result
shows convergence of the solutions of the ODE,
it is not clear
whether this implies the convergence
of the fixed-point iteration
$\xi_{k + 1} := \xi_k + \tau_k r(\xi_k)$
for some choice of the step sizes.
If $r$ would be differentiable,
one could use results from ODE theory
involving the eigenvalues of the Jacobian of $r$.
However, since $r$ is only $PC^1$,
it is not clear if these results carry over
to our situation.

We will see in the following that small constant step sizes are feasible in some situations.

\begin{theorem}
	\label{thm:gradient_method_constant_stepsize_R_linear}
	Let $\bar \xi \in \R^n$ be a zero of $r$ and $\bar u := \tilde u(\bar \xi)$.
	We assume that the matrix $\nabla^2 g(\bar u)$ is positive semidefinite.
	Then,
	there exist $\varrho, \tau_0 > 0$,
	such that for all $\xi_0 \in \R^n$ with $\abs{\xi_0 - \bar \xi} < \varrho$ and $\tau \in (0,\tau_0)$
	the sequence $\seq{\xi_k} \subset \R^n$ inductively defined via $\xi_{k+1} := \xi_k + \tau r(\xi_k)$ converges r-linearly towards $\bar \xi$.
\end{theorem}
\begin{proof}
	As $\nabla^2 g(\bar u)$ is symmetric (thus diagonizable) with non-negative eigenvalues,
	there exist
	an eigenvalue decomposition
	$\nabla^2 g(\bar u) = U^\top \Lambda U$,
	i.e.,
	$U \in \R^{n \times n}$ is orthogonal and
	$\Lambda = \diag(\lambda_1, \ldots, \lambda_n)$
	with eigenvalues $\lambda_1 \ge \ldots \ge \lambda_i > 0 = \lambda_{i+1} = \ldots = \lambda_n$ for some $i \in \set*{0, \ldots, n}$.
	We first consider the case $\lambda_1 > 0$
	which corresponds to $i \ge 1$.
	We define s.p.d.\ matrices $M \in \R^{i \times i}$ and $N \in \R^{n \times n}$
	with
	\[
		\Lambda
		=
		\begin{pmatrix}
			M & 0 \\
			0 & 0
		\end{pmatrix}
		,
		\qquad
		N
		:=
		U^\top
		\begin{pmatrix}
			M^{-1} & 0 \\
			0 & \kappa I
		\end{pmatrix}
		U
		,
	\]
	where $I$ is the $(n-i)$-dimensional identity matrix and
	$\kappa :=  \frac {\lambda_1} {2 \alpha^2} \max_{\xi \in \R^n} \norm{\Theta(\xi)}^2 + \frac 1 {2 \lambda_1}$.
	Note that $\kappa < \infty$ due to \itemref{lem:properties of Theta:ii}.
	Let $\hat \Lambda(\xi) \in \R^{n \times n}$ be the matrix such that
	$\nabla^2 g(\tilde u(\xi)) = U^\top \hat \Lambda(\xi) U$.

	We define the Lyapunov function
	$V(\xi) := \frac 1 2 \abs{\delta \xi}_N^2 := \frac 1 2 \delta \xi^\top N \delta \xi$
	where $\delta \xi := \xi - \bar \xi$.	

	We choose $\varrho > 0$ such that
	for all $\xi$ with $\abs{\xi-\bar \xi} < \varrho$
	we have
	$\frac 1 \alpha \norm{N U^\top (\hat \Lambda(\xi) - \Lambda) U \Theta(\xi)} \le \frac 1 {4 \lambda_1}$ and
	$\abs{N R(\delta \xi)} / \abs{\delta \xi} \le \frac 1 {8 \lambda_1}$,
	where $R$ is the remainder in the Taylor-like expansion of $r$ at $\bar \xi$,
	cf.\ \cref{lem:PC1-functions are semismooth}, i.e.,
	\[
		r(\xi)
		=
		0
		+
		\parens*{
			\frac 1 \alpha U^\top \hat \Lambda(\xi) U \Theta(\xi) - I
		}\delta \xi + R(\delta \xi)
		.
	\]
	Let $(U \delta \xi)_1$ be the first $i$ components of $U \delta \xi$ and $(U \delta \xi)_2$ the last $n-i$ components.
	We start with the bound
	\begin{align*}
		\MoveEqLeft
		\delta \xi^\top N \parens*{\frac 1 \alpha U^\top \hat \Lambda(\xi) U \Theta(\xi) - I} \delta \xi
		\\
		&=
		\delta \xi^\top N \parens*{\frac 1 \alpha U^\top \Lambda U \Theta(\xi) - I} \delta \xi
		+
		\delta \xi^\top N \frac 1 \alpha  U^\top (\hat \Lambda(\xi) - \Lambda) U \Theta(\xi) \delta \xi
		\\
		&\le
		\frac 1 \alpha \delta \xi^\top U^\top
		\begin{pmatrix}
			I & 0 \\
			0 & 0
		\end{pmatrix}
		U \Theta(\xi) U^\top U \delta \xi
		-
		\abs{\delta \xi}_N^2
		+ \frac {\abs{\delta \xi}^2} {4 \lambda_1} 
		\\
		&\le
		\frac 1 \alpha 
		\begin{pmatrix}
			(U \delta \xi)_1 \\
			0
		\end{pmatrix}
		^\top
		U \Theta(\xi) U^\top
		\begin{pmatrix}
			(U \delta \xi)_1 \\
			(U \delta \xi)_2
		\end{pmatrix}
		- \abs{\delta \xi}_N^2
		+ \frac {\abs{\delta \xi}^2} {4 \lambda_1} 
		.
	\end{align*}
	The matrix $U \Theta(\xi) U^\top$ is symmetric and negative
	semidefinite, see \itemref{lem:properties of Theta:i}, thus
	the upper component $(U \delta \xi)_1$ yields an upper bound
	of zero. We continue with
	\begin{align*}
		\MoveEqLeft
		\delta \xi^\top N \parens*{\frac 1 \alpha U^\top \hat \Lambda(\xi) U \Theta(\xi) - I} \delta \xi
		\\
		&\le
		0
		+ \frac 1 \alpha \norm{U \Theta(\xi) U^\top} \abs{(U \delta \xi)_1} \abs{(U \delta \xi)_2}
		- \frac { \abs{(U \delta \xi)_1}^2} {\lambda_1}
		- \kappa \abs{(U \delta \xi)_2}^2
		+ \frac {\abs{\delta \xi}^2} {4 \lambda_1} 
		\\
		&\le
		\left( \frac {\abs{(U \delta \xi)_1}^2} {2 \lambda_1}
		+ \frac{\lambda_1 \norm{\Theta(\xi)}^2 \abs{(U \delta \xi)_2}^2 }{2 \alpha^2} \right)
		- \frac {\abs{(U \delta \xi)_1}^2} {\lambda_1}
		- \kappa \abs{(U \delta \xi)_2}^2
		+ \frac {\abs{\delta \xi}^2} {4 \lambda_1}
		\\
		&=
		- \frac {\abs{(U \delta \xi)_1}^2} {2 \lambda_1}
		- \frac {\abs{(U \delta \xi)_2}^2} {2 \lambda_1}
		+ \frac {\abs{\delta \xi}^2} {4 \lambda_1}
		=
		-\frac{\abs{\delta \xi}^2}{4 \lambda_1}
		.
	\end{align*}
	In the third line, we used Young's inequality.
	Let $L$ be the Lipschitz constant of $r$
	and
	we recall $r(\bar \xi) = 0$.
	We set
	$\tau_0 := \frac 1 {4 \norm{N} \lambda_1 L^2}$.
	For $\tau < \tau_0$, we obtain
	\begin{align*}
		\MoveEqLeft
		V(\xi+\tau r(\xi)) - V(\xi)
		=
		\tau \delta \xi^\top N r(\xi)
		+
		\frac{\tau^2}{2} \abs{r(\xi)}_N^2
		\\
		&\le
		\tau \delta \xi^\top N
		\left( \left( \frac 1 \alpha U^\top \hat \Lambda(\xi) U \Theta(\xi) - I \right) \delta \xi
		+ R(\delta \xi) \right)
		+
		\frac{\tau^2}{2} \abs{r(\xi) - r(\bar \xi)}_N^2
		\\
		&\le
		\tau \left( \frac {\abs{N R(\delta \xi)}}{\abs{\delta \xi}} -\frac 1 {4\lambda_1}  \right)
		\abs{\delta \xi}^2
		+
		\frac{\tau^2}{2} \norm{N} L^2 \abs{\delta \xi}^2
		\le
		\left(\frac \tau 2 \norm{N} L^2  - \frac 1 {8 \lambda_1} \right)
		\tau \abs{\delta \xi}^2
		.
	\end{align*}
	The above choice of $\tau$
	yields that
	$\beta := (\frac 1 {8 \lambda_1} - \frac \tau 2 \norm{N} L^2) \tau > 0$.
	By the equivalence of norms in $\R^n$ we get the existence of $\vartheta > 0$ with
	$\abs{\delta \xi} \ge \vartheta \abs{\delta \xi}_N$.
	Thus, we obtain
	\[
		V(\xi + \tau r(\xi))
		\le
		V(\xi) - \beta \abs{\delta \xi}^2
		\le
		V(\xi) - \beta \vartheta^2 \abs{\delta \xi}_N^2
		=
		(1-2 \beta \vartheta^2) V(\xi)
	\]
	and, in turn, $\abs{\xi_{k+1} - \bar \xi}_N \le \sqrt{1-2\beta \vartheta^2} \abs{\xi_k - \bar \xi}_N$.
	This shows
	q-linear convergence in the $N$-norm
	and, consequently,
	r-linear convergence in the Euclidean norm.

	It remains to
	consider the case $\lambda_1 = 0$, i.e., $\nabla^2 g(\bar u) = 0$.
	We can choose the Lyapunov function $V(\xi) := \frac 1 2 \abs{\delta \xi}^2$.
	With arguments similar to those above,
	one can again check the convergence.
\end{proof}

\subsection{Semismooth Newton algorithm}
\label{subsec:Newton on r}

The semismooth Newton method aims to solve
the operator equation
\begin{equation}
	\label{eq:operator equation Semismooth Newton}
	G(x) = 0,
\end{equation}
where $G \colon X \to Y$ is a mapping between Banach spaces $X,Y$.
We say that $G$ is Newton differentiable
with derivative $DG \colon X \to \mathcal{L}(X, Y)$
at $\bar x \in X$
if
\begin{equation*}
	\frac{ \norm{ G(x) - G(\bar x) - DG(x)(x - \bar x)}_Y }
	{\norm{x - \bar x}_X}
	\to
	0
	\qquad
	\text{as } x \to \bar x
	.
\end{equation*}
Note that \cref{lem:PC1-functions are semismooth}
shows that $PC^1$ functions
are Newton differentiable
for a single-valued selection $D f$
from the generalized Jacobian $\partial^C f$.

Given an initial guess $x_0 \in X$,
we define the semismooth Newton iteration via
\begin{equation}
	\label{eq:SSN}
	x_{k + 1}
	:=
	x_k - DG(x_k)^{-1} G(x_k)
	.
\end{equation}
Here, we need invertibility of $DG(x_k)$.

\begin{theorem}[\texorpdfstring{\cite[Thm.~2.12]{OptimizationBanachspaces}}{}]
	\label{thm:Semismooth Newton converges q-superlinearly}
	Let $\bar x \in X$ be a solution of \eqref{eq:operator equation Semismooth Newton}.
	We assume that $G$ is Newton differentiable at $\bar x$
	with derivative $DG$.
	Further, we suppose that there exists $C > 0$
	such that
	$DG(x)$ is invertible with $\norm{DG(x)^{-1}} \le C$ for all $x$
	in a neighborhood of $\bar x$.
	Then, there exists $\delta > 0$, such that
	for all $x_0 \in X$ with $\norm{x - \bar x}_X \le \delta$,
	the sequence $\seq{x_k} \subset X$ defined via
	\eqref{eq:SSN}
	converges towards
	$\bar x \in X$ $q$-superlinearly.
\end{theorem}

Due to \cref{thm:J_r_tilde_u_W_PC1_semismooth},
the function $r$
is Newton differentiable on $\R^n$ with derivative
\begin{equation}
	\label{eq:Dr}
	Dr(\xi) := \frac 1 \alpha \nabla^2 g(\tilde u(\xi)) \Theta(\xi) - I
	.
\end{equation}
In order to apply the semismooth Newton method,
we have to show the uniform invertibility of $Dr(\xi)$.
This can be guaranteed in case $g$ is convex.

\begin{lemma}
	\label{lem:inverses S^xi in Newton method bounded}
	We assume that $g$ is convex.
	For all $\xi \in \R^n$, the matrix $Dr(\xi)$ is invertible.
	Further, it exists $C > 0$ such that $\norm{Dr(\xi)^{-1}} \le C$ for all $\xi \in \R^n$.
\end{lemma}
\begin{proof}
	We consider $-Dr(\xi) = \frac 1 \alpha \nabla^2 g(\tilde u(\xi)) (-\Theta(\xi)) + I$.
	The convexity of $g$ and \itemref{lem:properties of Theta:i} yield that
	both
	$\nabla^2 g(\tilde u(\xi))$ and $-\Theta(\xi)$ are symmetric and positive semidefinite.
	Consequently,
	\cite[Theorem~3.1]{Veselic_2015}
	implies that $Dr(\xi)$ is invertible and
	\begin{equation*}
		\norm{ Dr(\xi)^{-1} }
		\le
		1 + \frac1\alpha \norm{\nabla^2 g(\tilde u(\xi))} \norm{\Theta(\xi)}
		.
	\end{equation*}
	\itemref{lem:properties of Theta:ii} shows that $\Theta$ is bounded.
	The Hessian of $g$ is bounded as it is continuous by assumption on $g$ and $\tilde u(\xi) \in \Uad$ where $\Uad$ is compact.
\end{proof}

Owing to the results above,
we get the convergence of the semismooth Newton method
in the case that $g$ is convex.
We emphasize that we do not need strong convexity of $g$.
\begin{theorem}
	\label{thm:convergence of Newton method on r}
	We assume that $g$ is convex
	and let
	$\bar \xi \in \R^n$ with $r(\bar \xi) = 0$ be given.
	Then, there exists $\delta > 0$,
	such that for all $\xi_0 \in \R^n$ with $\abs{\xi_0 - \bar \xi} \le \delta$
	the sequence $\seq{\xi_k}_{k \in \N_0} \subset \R^n$ generated by
	\begin{equation*}
		\xi_{k + 1}
		:=
		\xi_k - Dr(\xi_k)^{-1} r(\xi_k)
	\end{equation*}
	converges $q$-superlinearly towards $\bar \xi$.
\end{theorem}

As usual,
the globalization of the semismooth Newton method
is a delicate issue.
In the numerical examples,
we use a line search.
This is motivated by the next lemma
which shows the Newton direction
is a non-ascent direction under certain circumstances.

\begin{lemma}
	\label{lem:Newton direction of r is descent on J}
	We assume that $g$ is convex
	and that
	$\xi \in \R^n$ satisfies \eqref{DC}.
	Then, the Newton direction $\delta \xi := - Dr(\xi)^{-1} r(\xi)$
	satisfies $J'(\xi; \delta \xi) \le 0$.
\end{lemma}
\begin{proof}
	\Cref{thm:J_r_tilde u_W_differentiable_for_DC} yields
	$J'(\xi;\delta \xi) = -\alpha r(\xi)^\top \Theta(\xi) Dr(\xi)^{-1} r(\xi)$.
	We check that the matrix $\Theta(\xi) Dr(\xi)^{-1}$
	is symmetric and positive semidefinite.
	The Sherman--Morrison--Woodbury formula
	yields
	$(I+UV)^{-1} = I - U (I + VU)^{-1} V$ if $I+UV$ is invertible,
	where $I$ is the identity matrix.
	Let $V$ be the square root of
	the symmetric and positive semidefinite matrix $-\Theta(\xi)$,
	see \itemref{lem:properties of Theta:i},
	and
	$U := \frac 1 {\alpha} \nabla^2 g(\tilde u(\xi)) V$.
	Therefore,
	\begin{align*}
		\Theta(\xi) Dr(\xi)^{-1}
		&=
		V^2 (I+UV)^{-1}
		=
		V^2 \left[ I - U (I + VU)^{-1} V \right]
		\\
		&=
		V \left[ I - VU (I + VU)^{-1} \right] V
		=
		V (I + VU)^{-1} V
		.
	\end{align*}
	The matrix $I+VU$ is symmetric positive definite
	and
	$V$ is symmetric.
	Consequently, $\Theta(\xi) Dr(\xi)^{-1}$
	is symmetric and positive semidefinite.
\end{proof}

In the case that \eqref{DC}
is not satisfied at $\xi$,
we only get the formula
$J'(\xi;\delta \xi) = -\alpha r(\xi)^\top \hat\Theta Dr(\xi)^{-1} r(\xi)$
for some $\hat\Theta$ from the generalized Jacobian of $\tilde u$,
see \cref{lem:PC1-functions:form of derivative}.
In general, this expression can be positive.
Note that
the points violating \eqref{DC} have measure zero,
see \cref{thm:degenerated sectors are zero-set}.
In our numerical computations in \cref{sec:numerics},
we never observed a situation in which the
Newton direction is an ascent direction for $J$.

\section{Numerical examples}
\label{sec:numerics}

We present some numerical results.
We address the optimal control problem \eqref{eq:OCP}.
In order to simplify the implementation,
we further restrict ourselves to
$u_d$ being the Lebesgue measure, cf.\ \cref{subsec:Lebesgue_case},
and
$\Omega \subset \R^2$ being an open and bounded polygon.

The state equation is discretized using finite elements.
That is,
we consider a triangulation of $\cl(\Omega)$
and the state $y$ and the test function $\varphi$ are approximated by
the space of
continuous and piecewise linear functions.
Further, we assume
that the control set $F$ is a subset of the nodes of the triangulation of $\cl(\Omega)$.
As described in \cref{subsec:notation_assumptions},
a control $u$ is identified with a vector in $\R^n$.

Now, the discretized state equation is given by
\(
	K y = B u,
\)
where $K$ is the stiffness matrix
reduced to the interior nodes of the triangulation
and $B$ is a matrix (containing only zeros and ones)
which relates the vertices $a_i \in F$
with the inner nodes of the triangulation.
Similarly, the tracking term is discretized by
\(
	\frac12 (y - y_d)^\top M (y - y_d)
\)
with the mass matrix $M$ (restricted to interior nodes)
and a discretization of the desired state $y_d$.

Unless stated otherwise,
we use the following problem data:
\begin{align*}
	\alpha &= 10^{-3}
	,
	\quad
	O = \Omega = (-1, 1)^2
	,
	\quad
	D = [-1,1]^2
	,
	\quad
	\beta = 0,
	\\
	y_d(x_1, x_2) &= 
	\frac 12 \cos(\pi x_1/2) \cos(\pi x_2/2) \exp(x_2)
	.
\end{align*}
The set $F$ itself will be the set of all nodes of the triangulation.

We stop our iterations as soon as
$\abs{r(\xi_k)}_\infty < 10^{-6}$ holds for some iterate $\xi_k \in \R^n$.
This is motivated by \cref{lem:optimality of xi on J}.

The Matlab source code for the computations
can be found in the GitHub repository
\url{https://github.com/gerw/transport_control},
see also \cite{BorchardWachsmuth2025:2}.

\subsection{Results for standard problem data}

We solve the discretized problem using
the algorithms presented in \cref{subsec:fixed_point} and \cref{subsec:Newton on r}.
We were not able to prove convergence of the fixed-point iteration in \cref{subsec:fixed_point},
but in the numerical experiments,
we did not observe any problems.
In our implementation of the fixed-point iteration
we used a strong Wolfe line search
(see \cite[Algorithmus~6.2]{GeigerKanzow1999})
to obtain the step size $\tau_k$.
Similarly, we used a strong Wolfe line search
to globalize the semismooth Newton method.
Although this is not covered by our theory,
both implementations work very reliably.

The solution of the discretized problem
(with the parameters given above)
on a grid with $8321$ nodes
is shown in \cref{fig:solution}.
\begin{figure}[ht]
	\centering
	\includegraphics[scale=0.09]{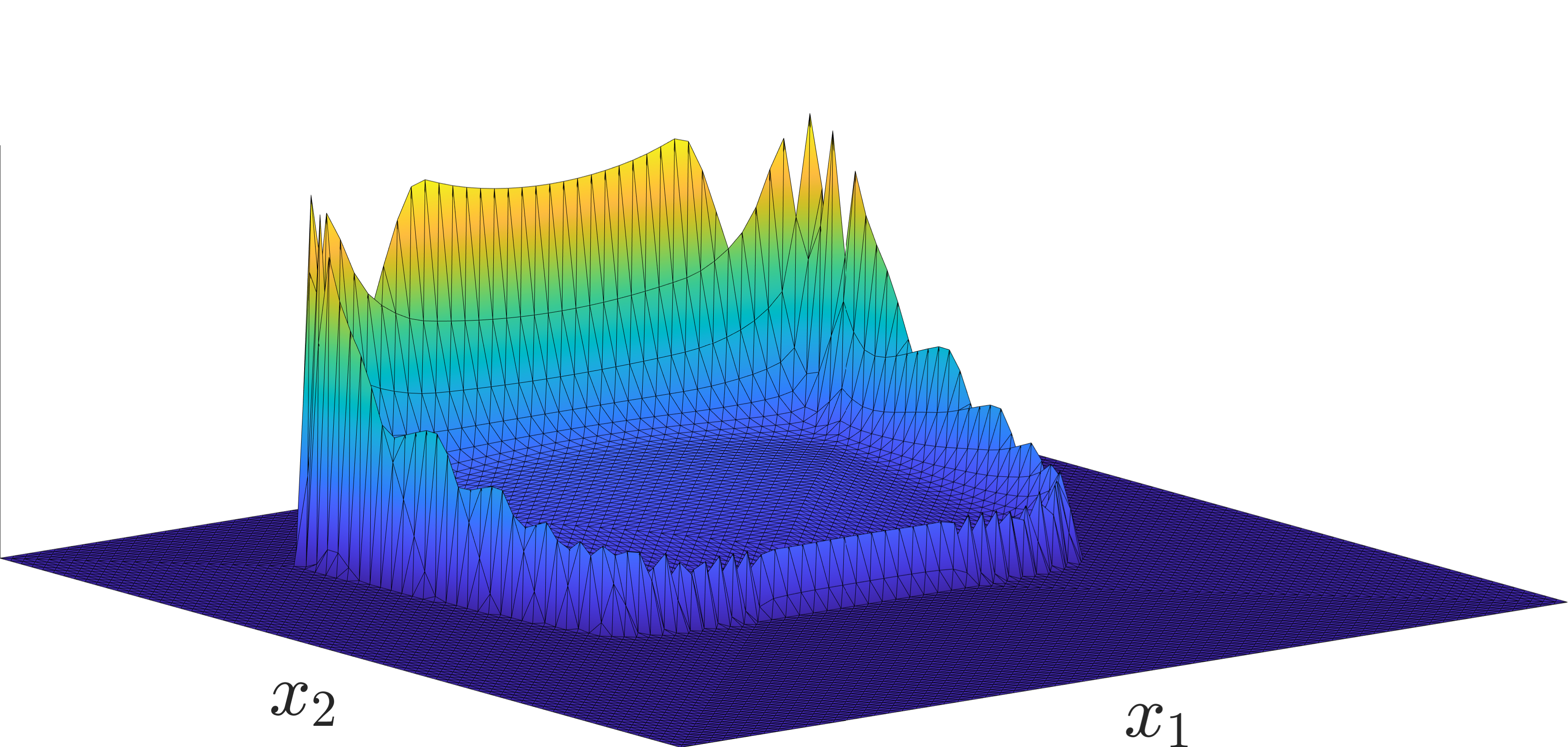}
	\qquad
	\includegraphics[scale=0.09]{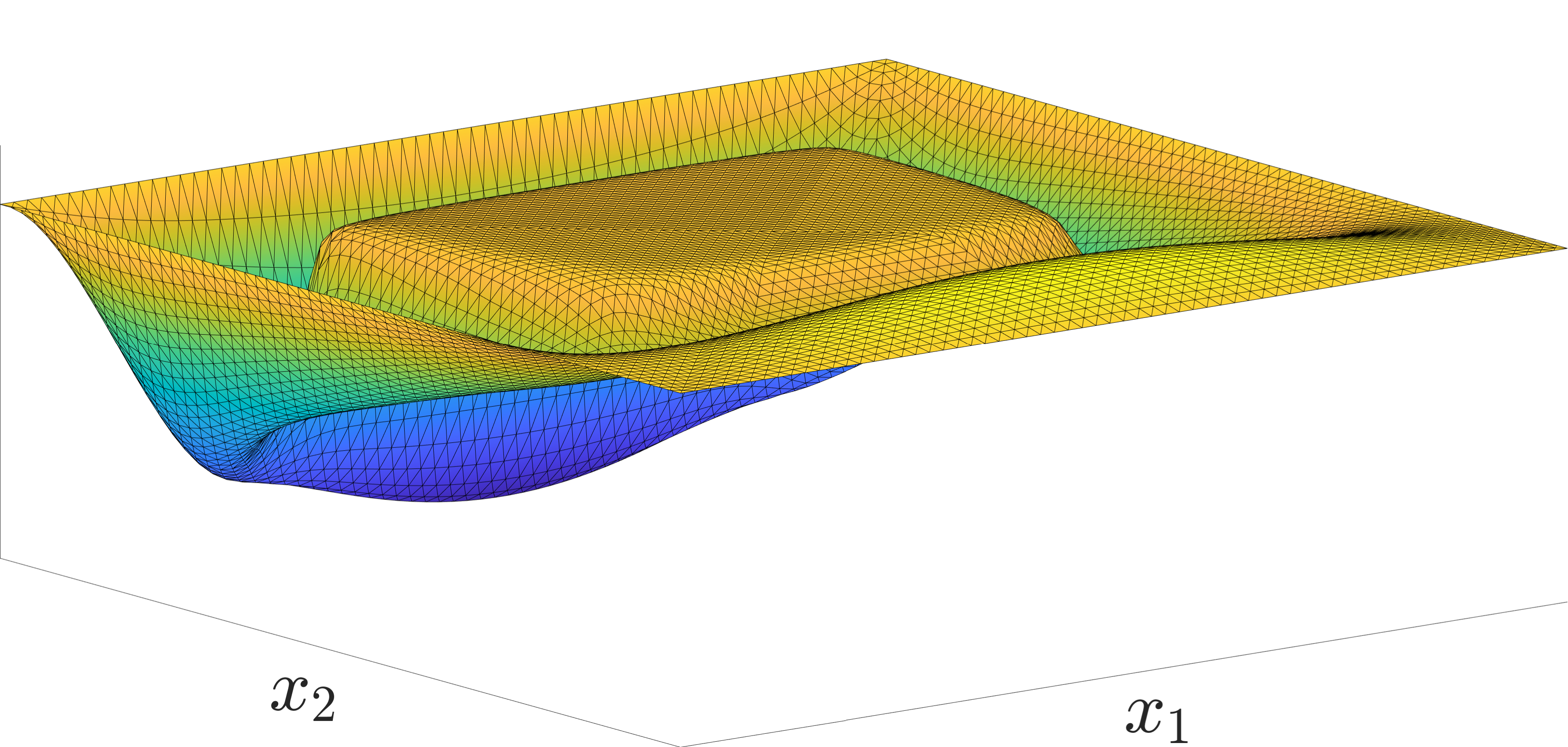}
	\caption{Optimal solution $\bar u$ (left) and residual $\bar y - y_d$ (right).}
	\label{fig:solution}
\end{figure}
We briefly explain the visualization
of the measure $\bar u$ in the left part of \cref{fig:solution}.
Recall that the discretized $\bar u$
is a linear combination of Dirac measures
in the nodes of the triangulation.
Since many coefficients are nonzero,
it is not enlightening to plot all the
individual Dirac measures.
Instead, we divide each component
$\bar u_i$
by $\int_\Omega \varphi_i \d x$,
where $\varphi_i$ is the nodal basis function
associated with the node $a_i$,
and we plot the resulting
piecewise linear function
$
\sum_{i = 1}^n
\bar u_i \parens{\int_\Omega \varphi_i \d x}^{-1}
\varphi_i
.
$

In \cref{fig:transport},
we
illustrate the transport from $u_d$ to the optimal $\bar u = \tilde u(\bar \xi)$
on a rather coarse mesh.
\begin{figure}[ht]
	\centering
	\includegraphics[align=c,scale=0.11]{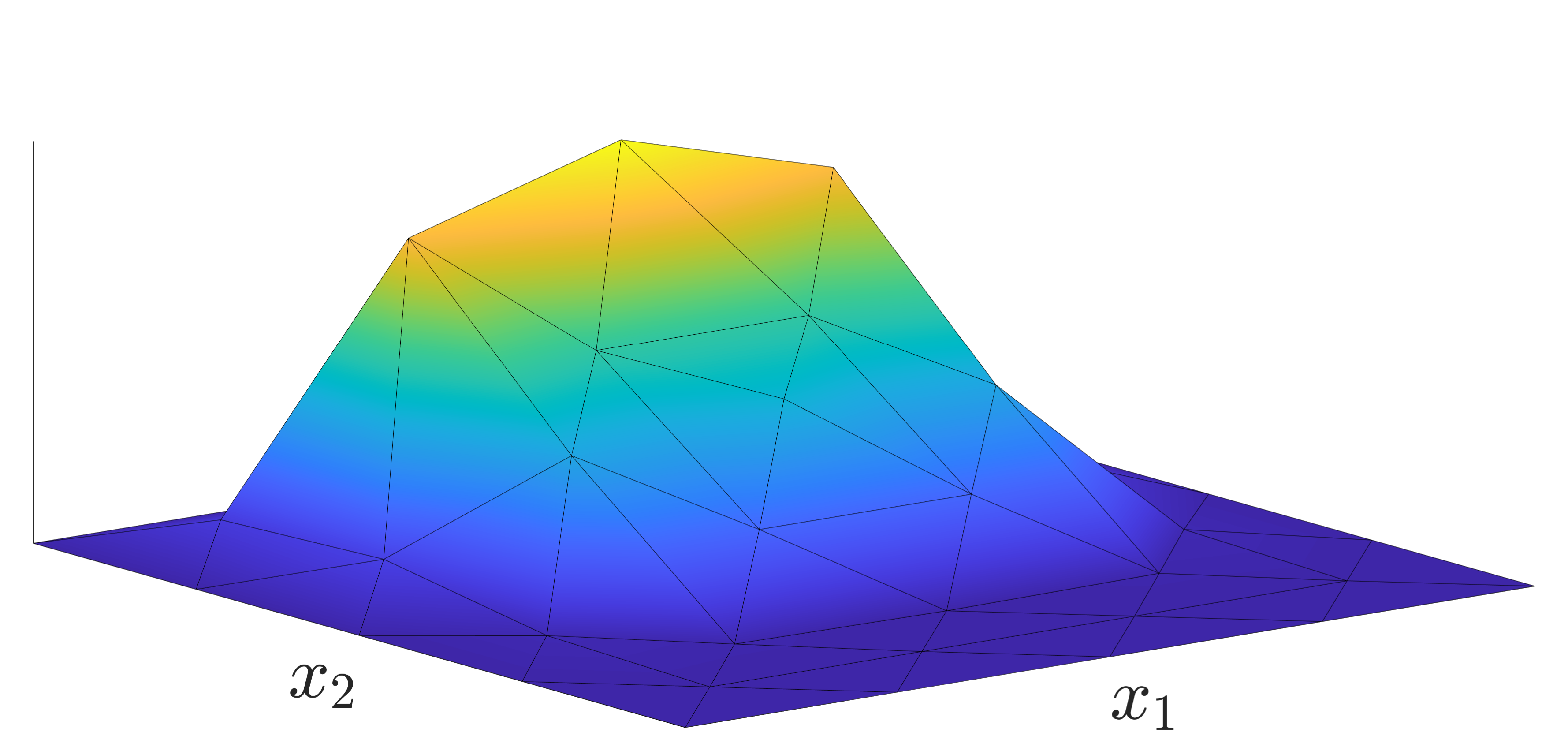}
	\qquad
	\includegraphics[align=c,scale=0.16]{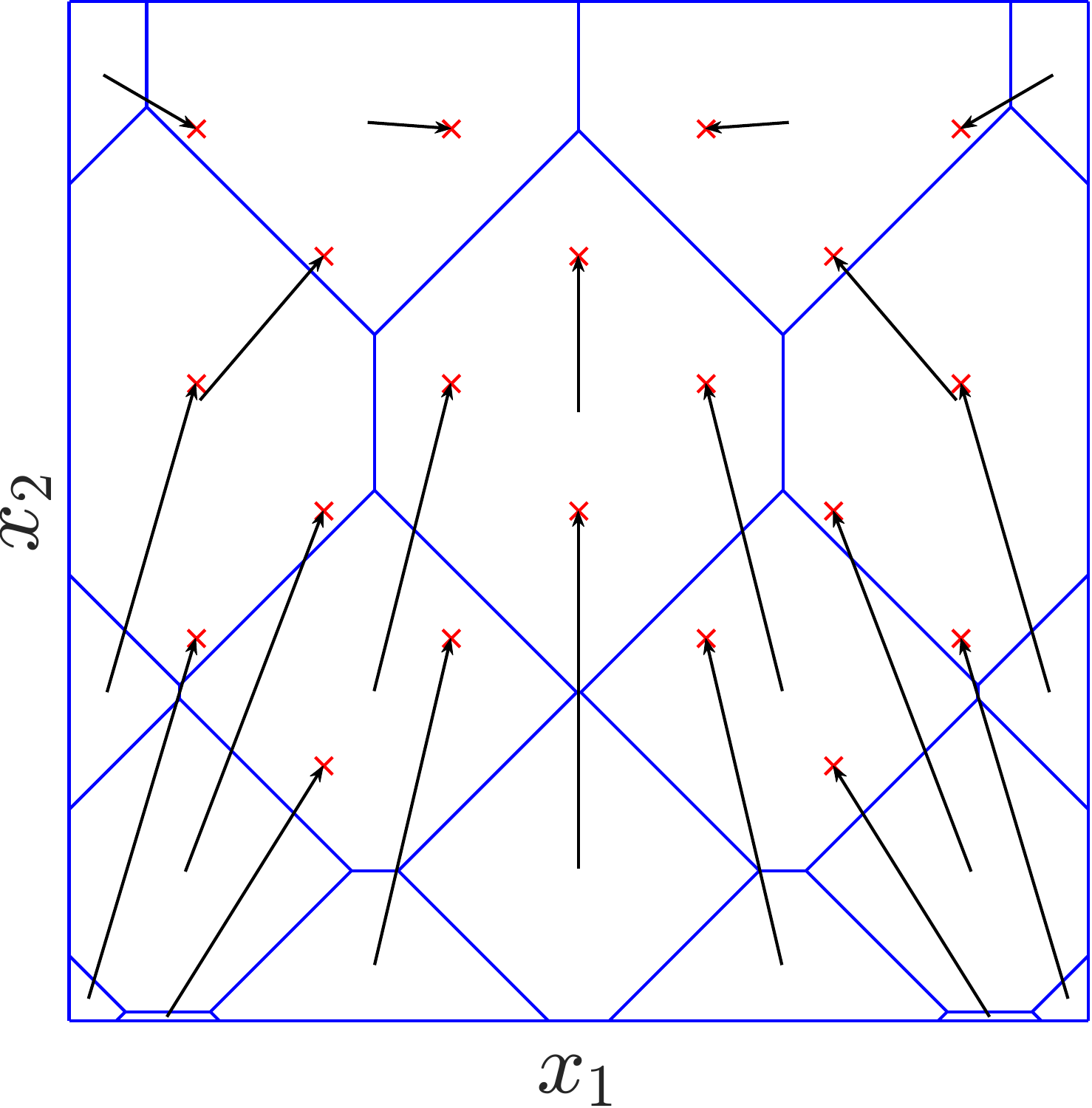}
	\caption{Optimal solution $\bar u$ and corresponding transport plan on a coarse mesh.}
	\label{fig:transport}
\end{figure}
In the right part of this figure,
one can see the nonempty dominating regions $D_i(\bar \xi)$
(blue polygons)
and the corresponding arrow points to the associated $a_i$
(red crosses).
Recall that the mass of $u_d$ located in $D_i(\bar \xi)$
is transported to $a_i$.

Next,
we investigate the performance of our two algorithms
on different refinements of the domain.
The results are reported in \cref{tab:algorithms}.
\begin{table}[ht]
	\centering
	\begin{tabular}{rr|rrr|rrr}
		\toprule
		\multirow{2}{*}{Ref} &
		\multirow{2}{*}{\# nodes} &
		\multicolumn{3}{c|}{fixed-point method} &
		\multicolumn{3}{c}{semismooth Newton method} \\
		&& 
		iter & time [s] & avg time [s] & iter & time [s]& avg time [s] \\
		\midrule
		4 &        545 &        532 &       2.50 & 4.704e$-$03 &         11 &       0.07 & 6.244e$-$03 \\
		5 &       2113 &        536 &       8.29 & 1.547e$-$02 &         13 &       0.33 & 2.519e$-$02 \\
		6 &       8321 &        542 &      33.47 & 6.175e$-$02 &         13 &       1.65 & 1.270e$-$01 \\
		7 &      33025 &        517 &     138.21 & 2.673e$-$01 &         12 &       9.20 & 7.670e$-$01 \\
		8 &     131585 &        574 &     678.78 & 1.183e$+$00 &         12 &      58.74 & 4.895e$+$00 \\
		\bottomrule
	\end{tabular}
	\caption{Number of iterations, total calculation time and average time (per iteration) for different refinements of a coarse mesh.}
	\label{tab:algorithms}
\end{table}
As expected, the fixed-point algorithm uses a big number of steps, however, these can be calculated rather fast.
On the other hand, the Newton steps take more time but only a small number of steps is required to obtain the same accuracy.

For both methods, the number of iterations is roughly constant
and does not depend on the discretization level.
This is also supported by \cref{fig:rinfnorm},
which shows how the infinity norm of the residual $r(\xi_k)$ decreases over the iterations $k$
of the semismooth Newton method
for different discretization levels.
\begin{figure}[ht]
	\centering
	\includegraphics[scale=0.2]{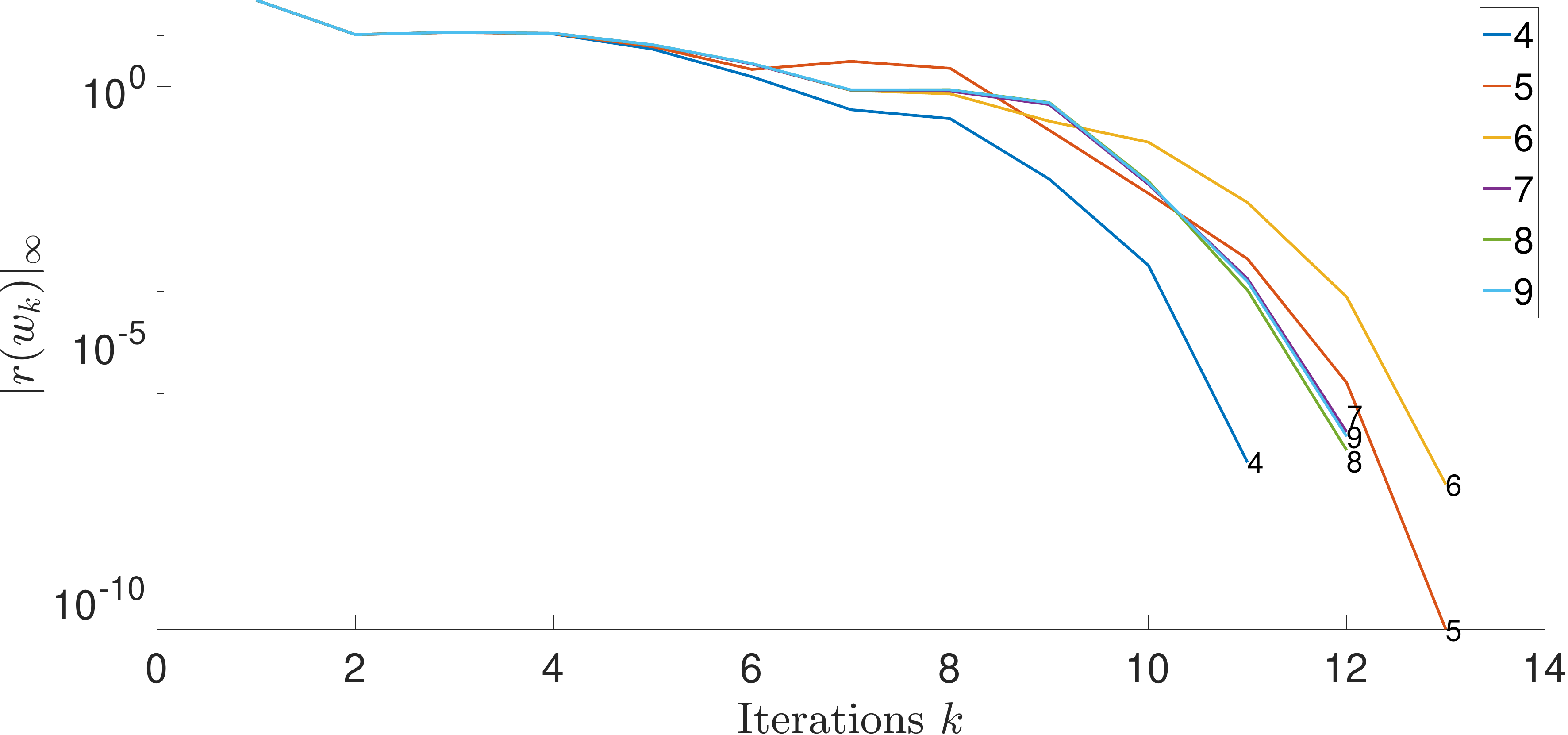}
	\caption{%
	Plot of $\abs{r(\xi_k)}_\infty$ for the semismooth Newton method over the iterations for different numbers of refinements.}
	\label{fig:rinfnorm}
\end{figure}
This observed mesh independence
is rather surprising,
since our analysis depends
crucially on the finite-dimensional setting
and, in particular,
the theory of $PC^1$ functions
which is not available
in infinite dimensions.

We shortly want to discuss the gradient algorithm for constant but small step sizes as it has been discussed in \cref{thm:gradient_method_constant_stepsize_R_linear}.
Of course,
this is highly dependent on the concrete value of the step size $\tau > 0$.
For example, for the choice $\tau = 10^{-2}$ we observe convergence
while
$\tau = 2\cdot10^{-2}$ results in divergence.
This behaviour seems to be independent of the level of discretization.
Even with small values of $\tau > 0$,
the function value does not necessarily decrease in every iteration,
but this does not impede convergence of the sequence.
We emphasize that we even observe global convergence
although our theoretical result only provides local convergence.
This is subject to future research.

\subsection{Convection-diffusion problem}
We present numerical examples
for the motivation problem from the introduction.
To this end,
we change the problem data to
$\Omega = (0,1)^2$,
$\beta = (16,32)$,
$y_d \equiv 0$.
The observation domain $O$ is the square with the vertices
$(0.55, 0.7)$, $(0.55, 0.8)$, $(0.65, 0.8)$, and $(0.65, 0.7)$.
The domain $D$ of the prior is the square with the vertices
$(0.4,0.5)$, $(0.3,0.4)$, $(0.4,0.3)$, and $(0.5,0.4)$.
In \cref{fig:convection_diffusion}
we depict the optimal state $y$
and (the density of) the optimal control $u$
for different values of $\alpha$.
\begin{figure}[ht]
	\centering
	\includegraphics[scale=0.28]{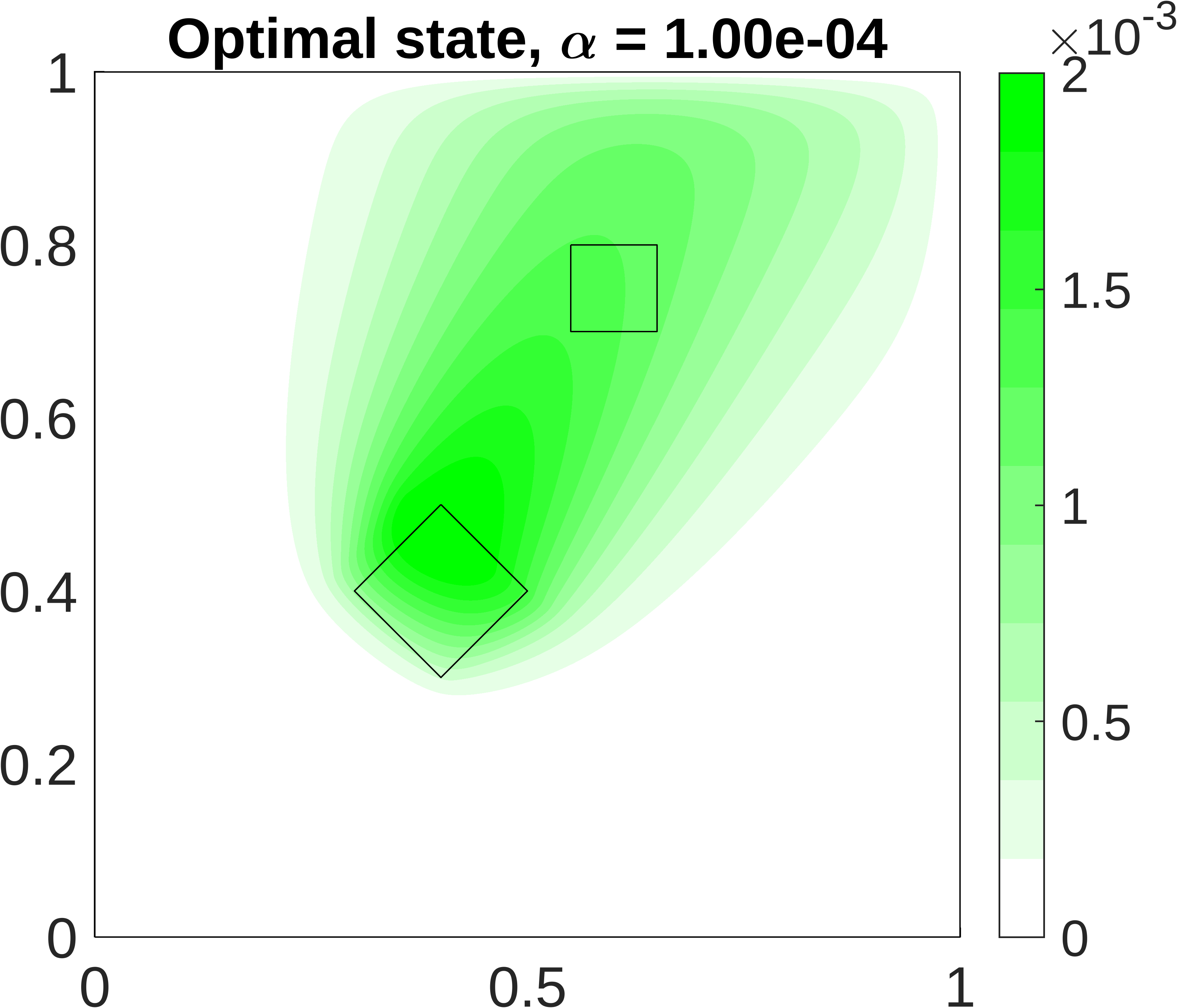}%
	\hfill%
	\includegraphics[scale=0.28]{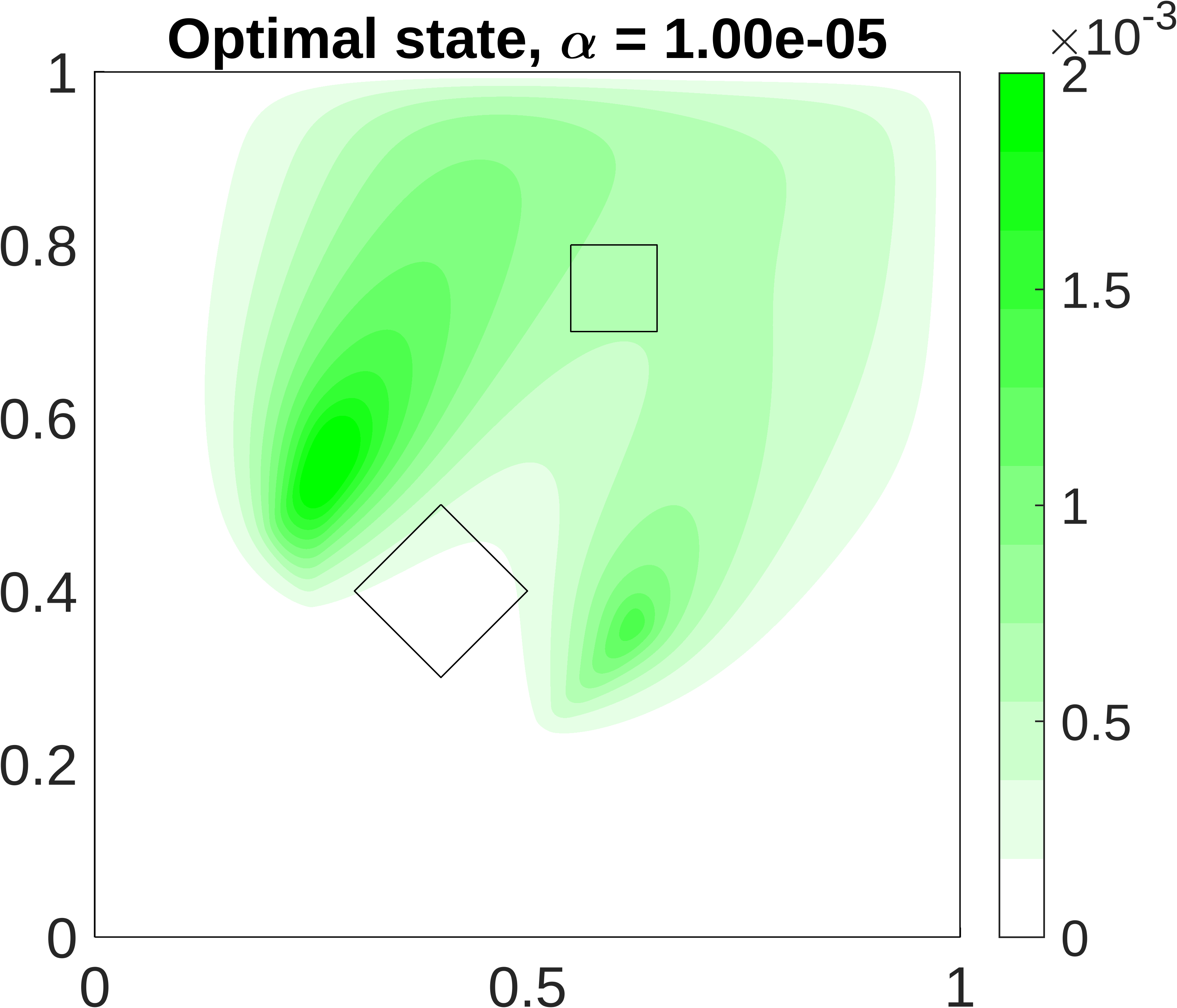}%
	\hfill%
	\includegraphics[scale=0.28]{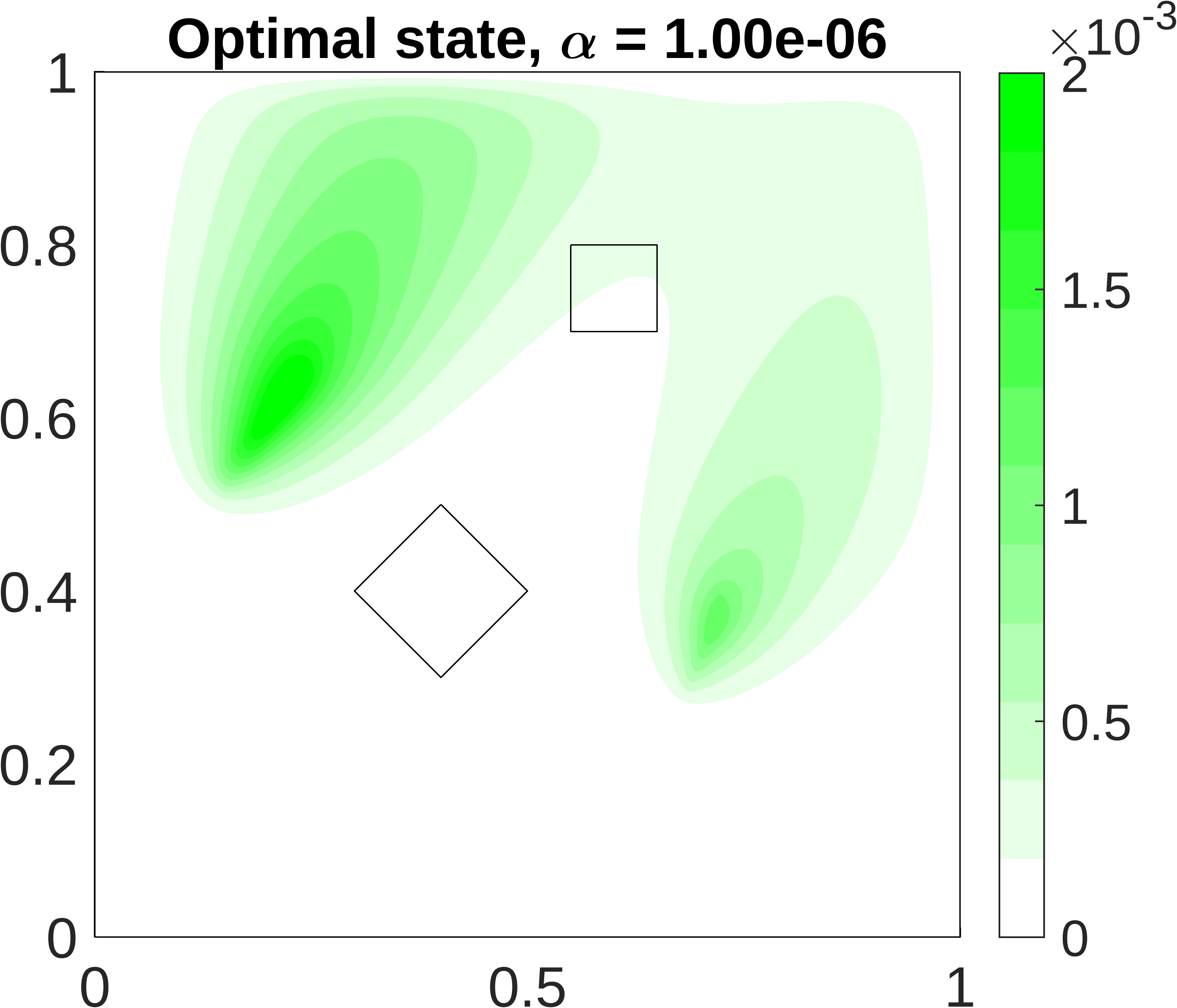}\\[.5cm]
	\includegraphics[scale=0.28]{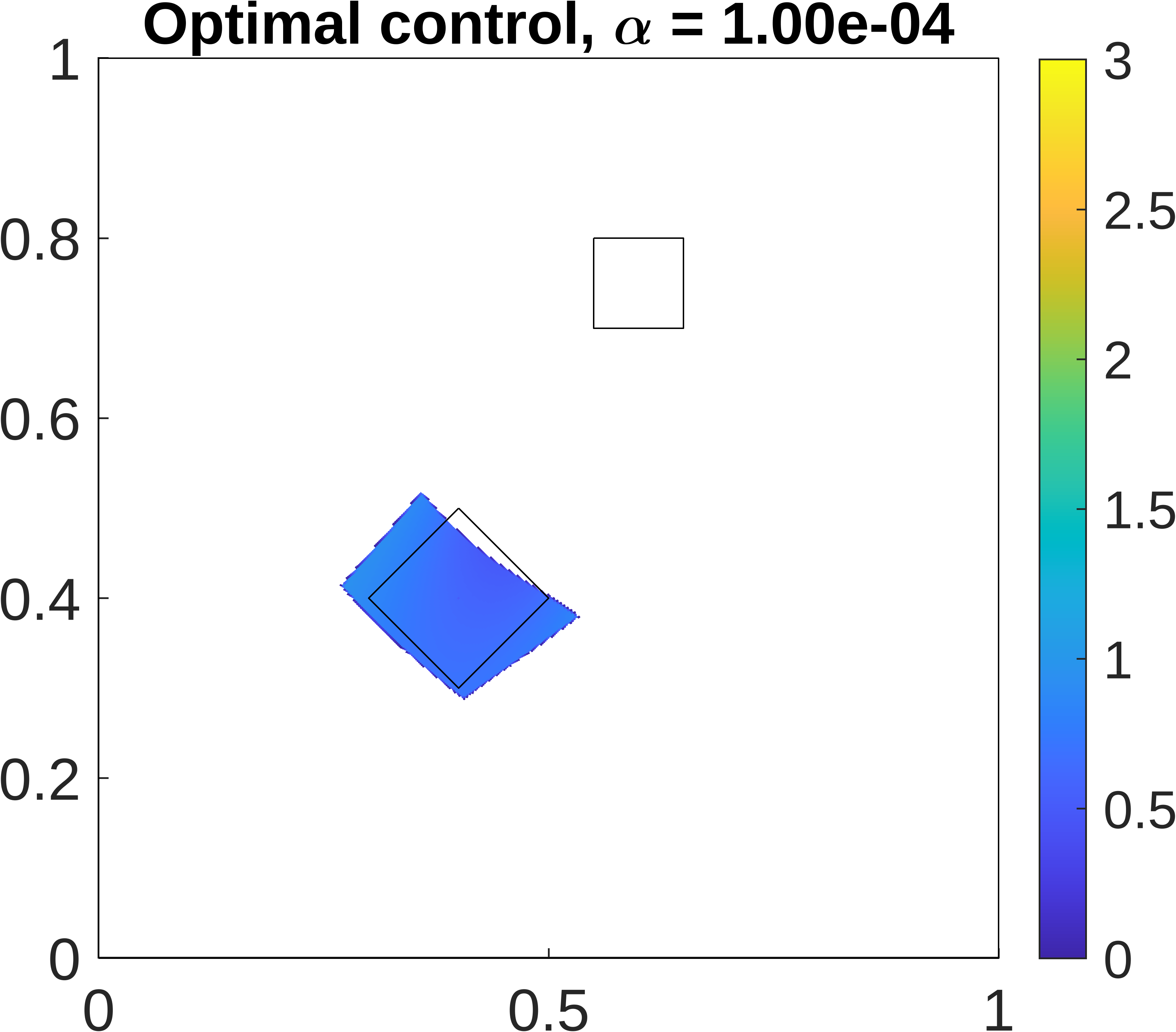}%
	\hfill%
	\includegraphics[scale=0.28]{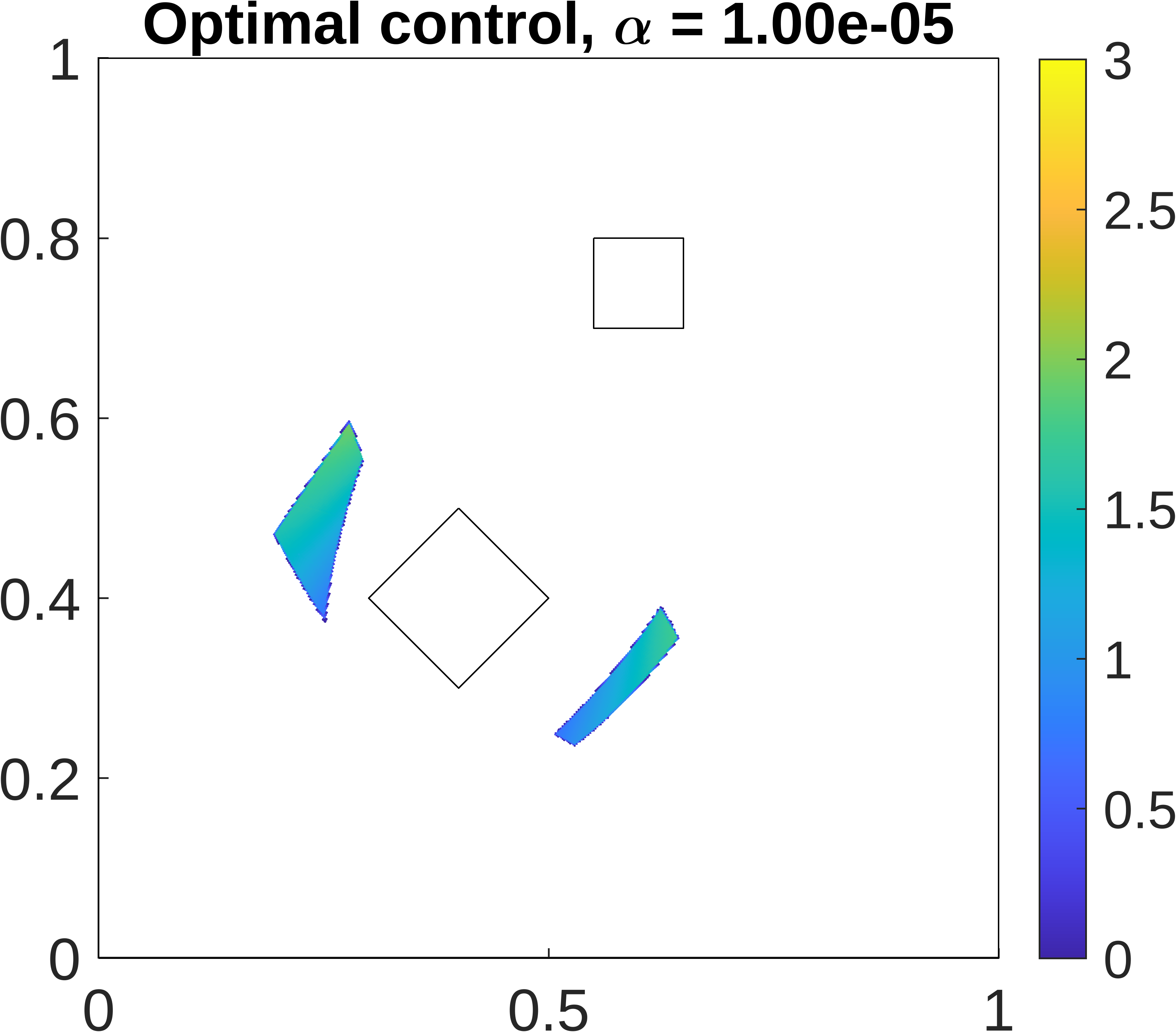}%
	\hfill%
	\includegraphics[scale=0.28]{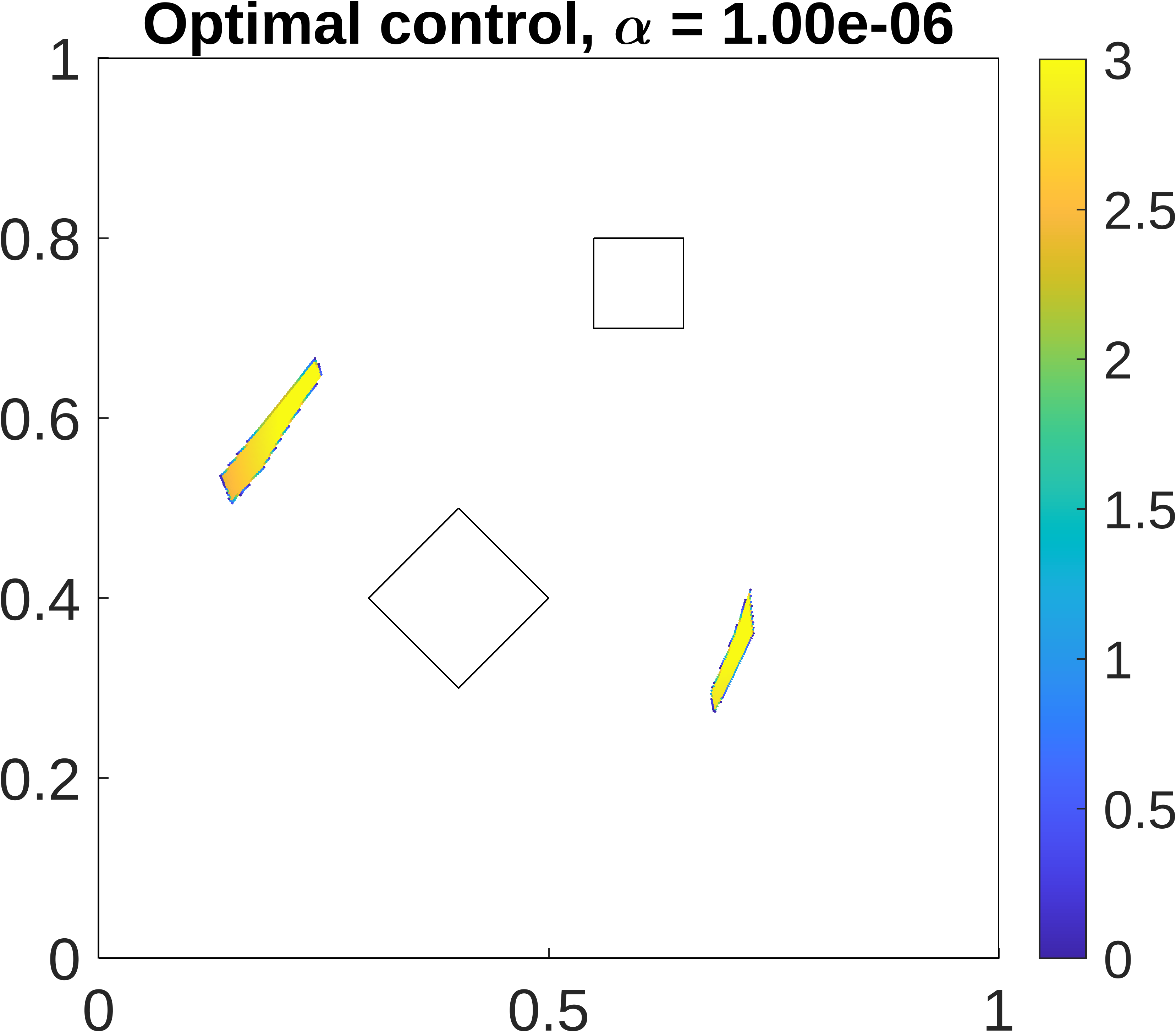}\\
	\caption{Optimal state and control for the convection-diffusion problem for various values of $\alpha$.}
	\label{fig:convection_diffusion}
\end{figure}
We see that for large values of $\alpha$,
the optimal control coincides almost with the prior $u_d$,
whereas for small values of $\alpha$,
the control $u$ is moved away from $D$
and
the state variable is very small on the observation domain $O$.

\subsection{Control on a subdomain}

We also implemented the case that the control is only non-zero on the subset $[0,1]^2$,
i.e.,
for the control set $F$ we choose all nodes of the triangulation that lie inside the top-right subsquare $[0,1]^2$.
In \cite{MeyerWachsmuth2025},
we prove that
under certain assumptions,
the optimal solution $\bar u$
can only contain Diracs on the boundary of the control set $F$.
Our
numerical examples,
see, e.g., \cref{fig:subset_control},
support this finding.
It seems that a Dirac measure is located in the vertex $(0,0)$
while a line measures is supported on the boundary segments of
$[0,1]^2$.
\begin{figure}[ht]
	\centering
	\includegraphics[scale=0.18]{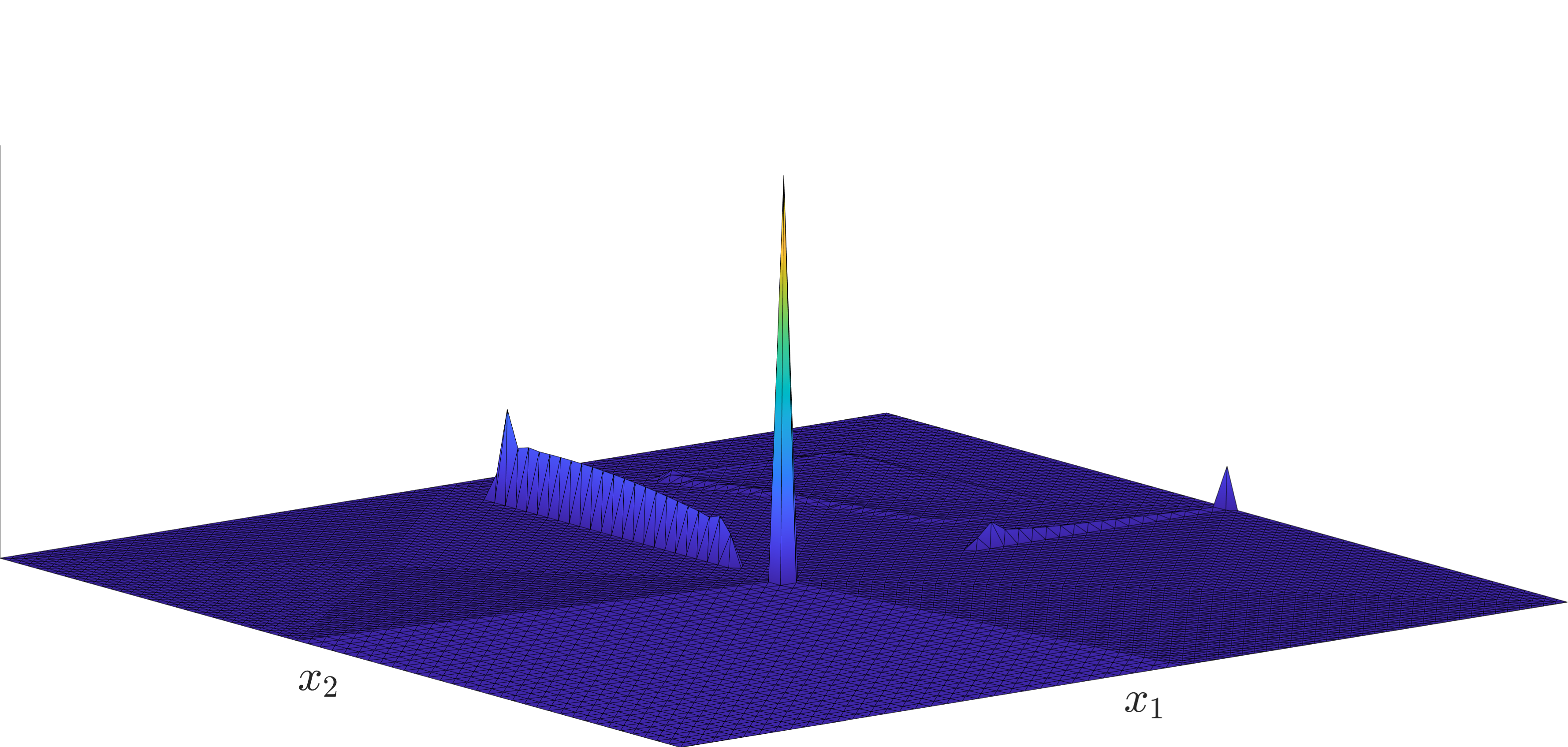}
	\caption{Optimal solution $\bar u$ for control only on the top-right subsquare.}
	\label{fig:subset_control}
\end{figure}

\subsection{Nonconvex objective}

We now consider a nonconvex $g$.
In particular,
we use
\begin{equation*}
	g(u)
	:=
	\left( \frac12 \norm{y_u - y_d}_{L^2(\Omega)}^2 \right)
	\left( \frac12 \norm{y_u - y_{d,2}}_{L^2(\Omega)}^2 \right)
	,
\end{equation*}
i.e., the product of two different tracking terms.
The function $y_d$ is chosen as above
and for the second desired state
we use
\[
	y_{d,2}(x_1, x_2)
	:=
	\frac 1 4
	\cos(\pi x_1 / 2 + \pi)
	\sin(\pi x_2)
	.
\]

Numerical examples show that our fixed-point algorithm (with line search) works in the nonconvex case
as well.
The Newton algorithm in \cref{thm:convergence of Newton method on r} also works for some examples
although we assumed convexity of $g$ in \cref{thm:convergence of Newton method on r},
again with a line search as globalization.
Note that the Newton differentiability
of the residual $r$
also holds in the nonconvex case.
However, the invertibility of the derivative
is not clear, cf.\ \cref{lem:inverses S^xi in Newton method bounded}.
The fixed-point algorithm for constant step sizes, cf.\ \cref{thm:gradient_method_constant_stepsize_R_linear}, seems to convergence for suitable step sizes as well.

In the convex case,
\cref{lem:optimality of xi on J}
provides the existence of
a unique zero of $r$, which corresponds to a minimizer.
For the above nonconvex $g$,
we numerically observed two different zeros,
see \cref{fig:nonconvex}.
\begin{figure}[ht]
	\centering
	\includegraphics[scale=0.09]{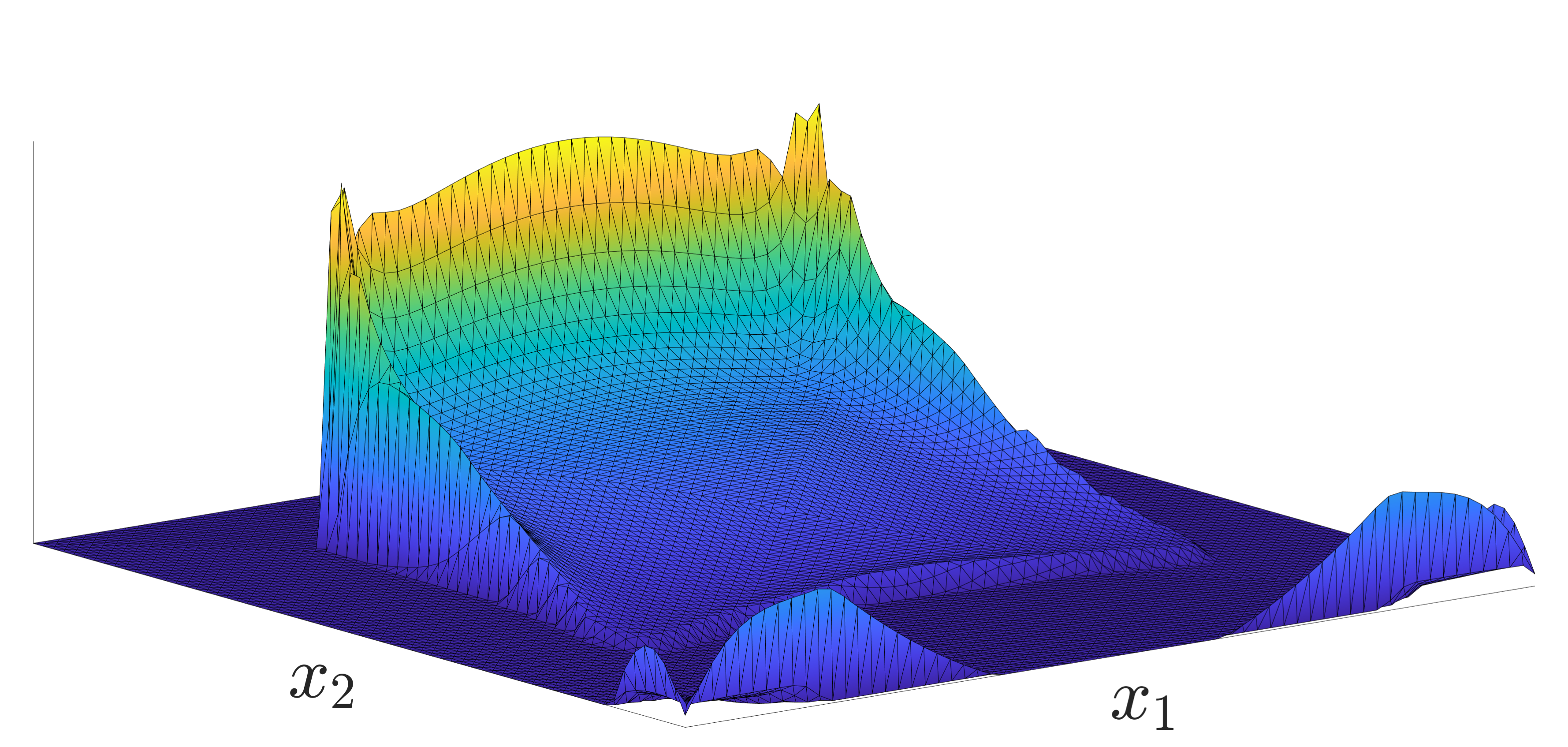}
	\qquad
	\includegraphics[scale=0.09]{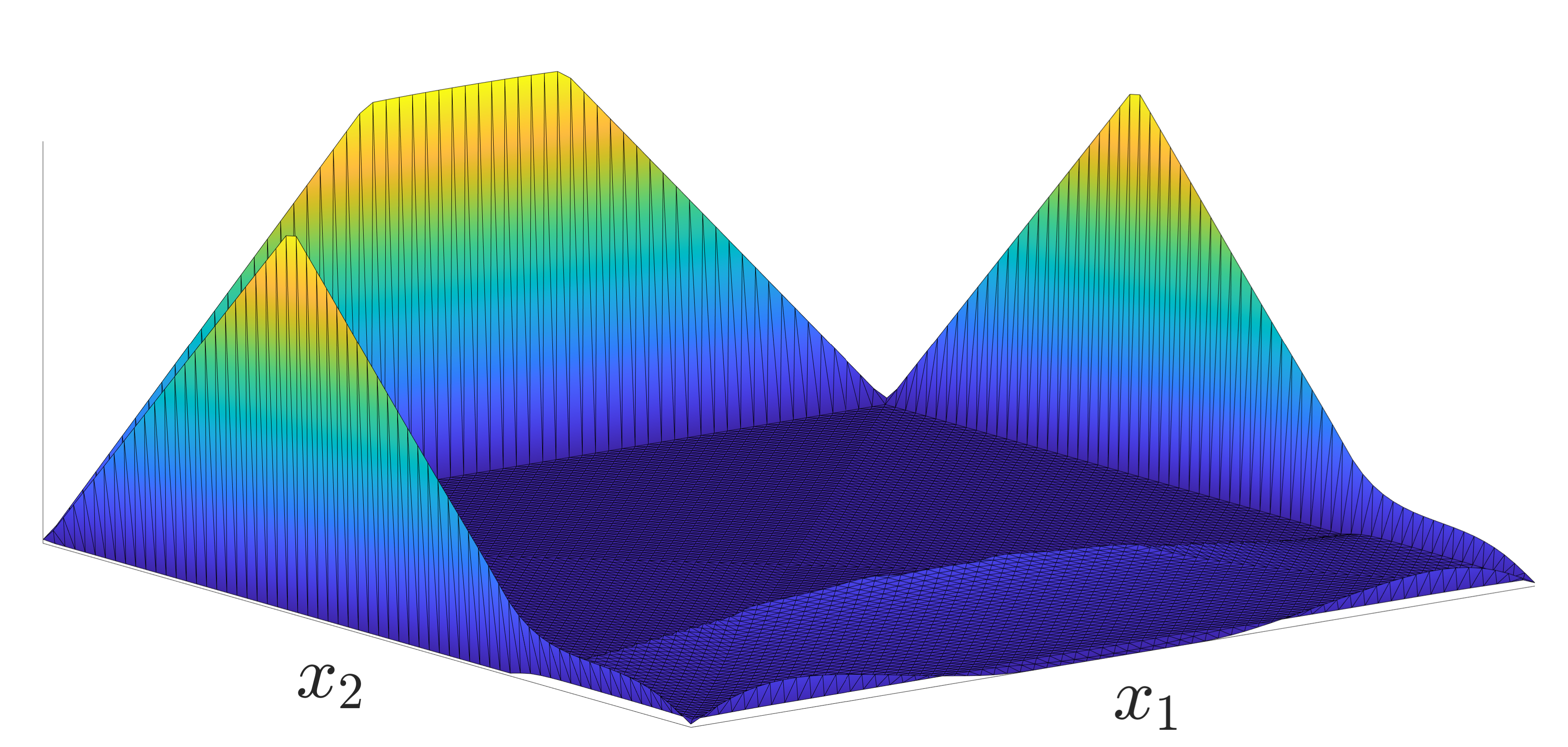}
	\includegraphics[scale=0.09]{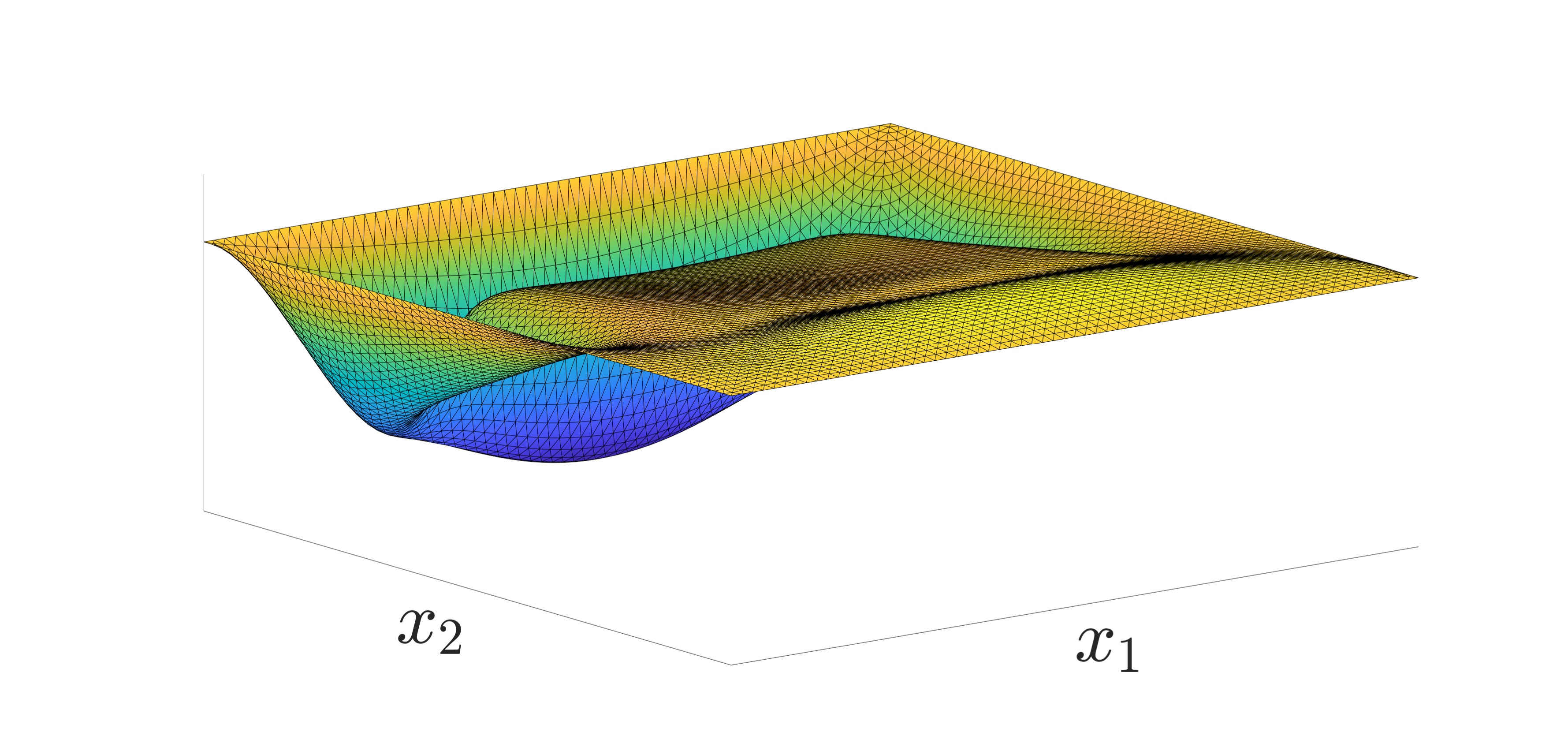}
	\includegraphics[scale=0.09]{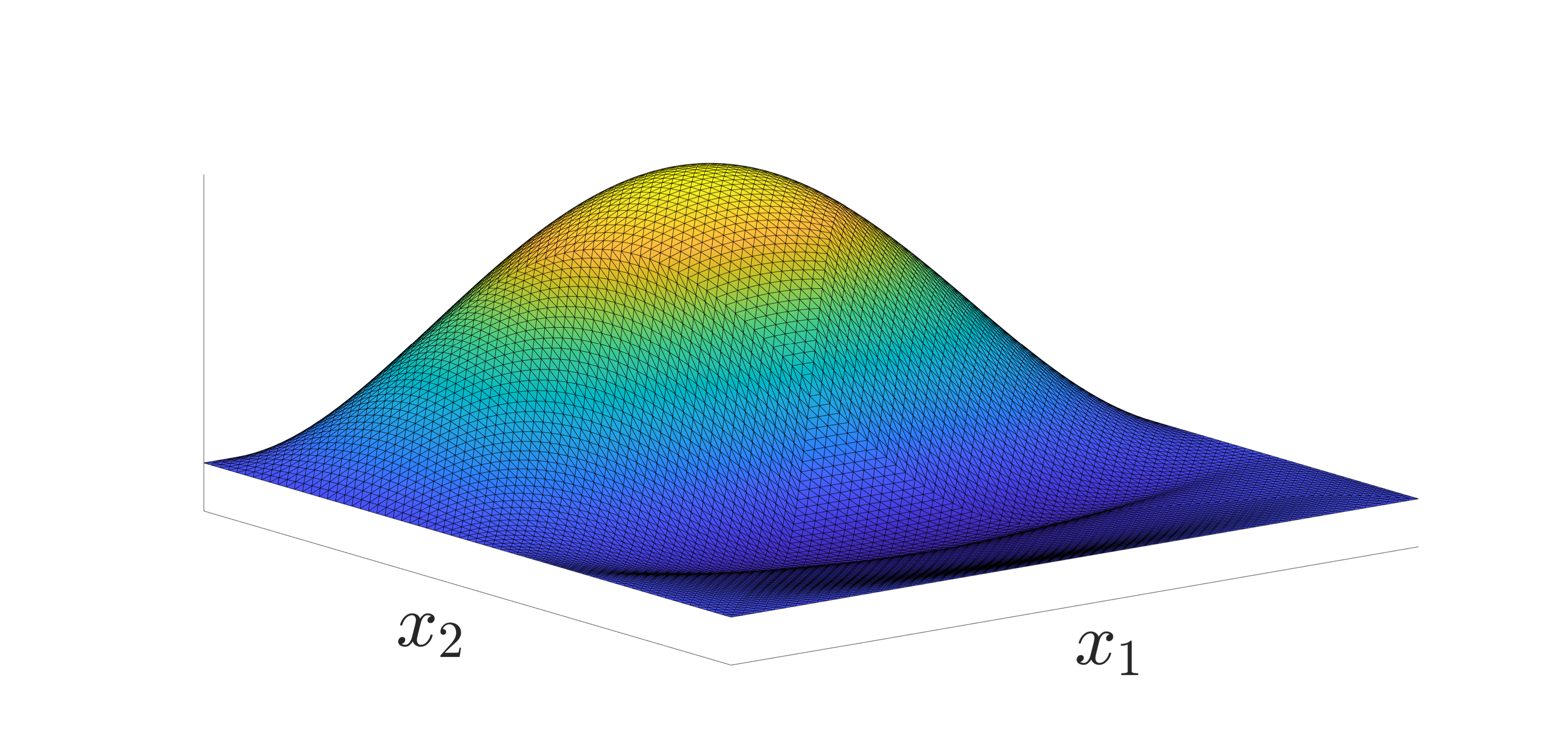}
	\caption{Two local solutions in the nonconvex case.
	The top row shows the optimal control whereas the bottom row shows the residuals.}
	\label{fig:nonconvex}
\end{figure}
The plots in this figure show
the two solutions $\bar u_1$ and $\bar u_2$ (top row)
as well as
the residuals $\bar y_1 - y_{d}$ and $\bar y_2 - y_{d,2}$ (bottom row).

\section*{Acknowledgments}
The authors would like to thank the anonymous referees
for the many helpful remarks.
\bibliographystyle{siamplain}
\bibliography{references}

\end{document}